\documentclass[final,12pt]{amsart}
    \title[Co-t-structures for the Kronecker algebra]{Classification of co-slicings and co-t-structures for the Kronecker algebra}
    
\usepackage{amssymb}
\usepackage{latexsym}
\usepackage{amsmath}
\usepackage{mathrsfs}
\usepackage[all]{xy}
\usepackage{enumerate}	
\usepackage[shortlabels]{enumitem}

\usepackage[notref, notcite]{showkeys}

\usepackage{tikz}
	\usetikzlibrary{decorations.pathmorphing}
	\usetikzlibrary{decorations.pathreplacing,calc}
	\usetikzlibrary{patterns}

    \tikzset{
        hatch distance/.store in=\hatchdistance,
        hatch distance=2.5pt,
        hatch thickness/.store in=\hatchthickness,
        hatch thickness=0.1pt
    }

    \makeatletter
    \pgfdeclarepatternformonly[\hatchdistance,\hatchthickness]{A hatch}
    {\pgfqpoint{0pt}{0pt}}
    {\pgfqpoint{\hatchdistance}{\hatchdistance}}
    {\pgfpoint{\hatchdistance}{\hatchdistance}}%
    {
        \pgfsetcolor{\tikz@pattern@color}
        \pgfsetlinewidth{\hatchthickness}
        \pgfpathmoveto{\pgfqpoint{0pt}{0pt}}
        \pgfpathlineto{\pgfqpoint{5pt}{0pt}}
        \pgfusepath{stroke}
    }

    \tikzset{
        hatch distance/.store in=\hatchdistance,
        hatch distance=2.5pt,
        hatch thickness/.store in=\hatchthickness,
        hatch thickness=0.1pt
    }

    \makeatletter
    \pgfdeclarepatternformonly[\hatchdistance,\hatchthickness]{B hatch}
    {\pgfqpoint{0pt}{0pt}}
    {\pgfqpoint{\hatchdistance}{\hatchdistance}}
    {\pgfpoint{\hatchdistance-1pt}{\hatchdistance-1pt}}%
    {
        \pgfsetcolor{\tikz@pattern@color}
        \pgfsetlinewidth{\hatchthickness}
        \pgfpathmoveto{\pgfqpoint{0pt}{0pt}}
        \pgfpathlineto{\pgfqpoint{0pt}{5pt}}
        \pgfusepath{stroke}
    }

\theoremstyle{plain}
\newtheorem{theorem}{Theorem}[section]

\newtheorem{lemma}[theorem]{Lemma}
\newtheorem{corollary}[theorem]{Corollary}
\newtheorem{proposition}[theorem]{Proposition}

\newtheorem{introtheorem}{Theorem}

\theoremstyle{definition}
\newtheorem{remark}[theorem]{Remark}
\newtheorem{example}[theorem]{Example}
\newtheorem{definition}[theorem]{Definition}
\newtheorem{definitions}[theorem]{Definitions}

\newtheorem{notation}[theorem]{Notation}

\newtheorem{triangles}[theorem]{Standard triangles}
\newtheorem{facts}[theorem]{Standard facts}

\DeclareMathAlphabet{\mathpzc}{OT1}{pzc}{m}{it}

\newcommand{\cE}{\mathscr{E}}

\newcommand{\cO}{\mathscr{O}}

\newcommand{\cQ}{\mathscr{Q}}
\newcommand{\cR}{\mathscr{R}}

\newcommand{\sA}{\mathsf{A}}
\newcommand{\sB}{\mathsf{B}}
\newcommand{\sC}{\mathsf{C}}
\newcommand{\sD}{\mathsf{D}}

\newcommand{\sK}{\mathsf{K}}

\newcommand{\sS}{\mathsf{S}}
\newcommand{\sT}{\mathsf{T}}

\newcommand{\bC}{\mathbb{C}}

\newcommand{\bK}{\mathbb{K}}

\newcommand{\bN}{\mathbb{N}}

\newcommand{\bP}{\mathbb{P}}

\newcommand{\bR}{\mathbb{R}}

\newcommand{\bZ}{\mathbb{Z}}

\renewcommand{\geq}{\geqslant}
\renewcommand{\leq}{\leqs}
\renewcommand{\phi}{\varphi}
\renewcommand{\epsilon}{\varepsilon}

\newcommand{\coker}{\textnormal{coker\,}}

\newcommand{\Hom}[3]{\mathsf{Hom}_{#1}(#2,#3)}

\newcommand{\Ext}[2]{\mathsf{Ext}^{#1}_{#2}}

\renewcommand{\mod}[1]{\mathsf{mod}(#1)}

\newcommand{\add}[1]{\mathsf{add}(#1)}
\newcommand{\ind}[1]{\mathsf{ind}\, #1}

\newcommand{\thick}[2]{\mathsf{thick}_{#1}(#2)}

\newcommand{\Costab}[1]{\mathsf{Costab}(#1)}
\newcommand{\Coslice}[1]{\mathsf{Coslice}(#1)}
\renewcommand{\ss}[1]{\mathsf{ss}(#1)}

\newcommand{\tri}[3]{#1\rightarrow #2\rightarrow #3\rightarrow \Sigma #1}

\newcommand{\trilabels}[6]{#1\stackrel{#4}{\longrightarrow} #2\stackrel{#5}{\longrightarrow} #3\stackrel{#6}{\longrightarrow} \Sigma #1}

\newcommand{\leqs}{\leqslant}

\entrymodifiers={+!!<0pt,\fontdimen22\textfont2>}

\begin{document}

\begin{abstract}
In this paper we introduce the notion of a `generalised' co-slicing of a triangulated category. This generalises the theory of co-stability conditions in a manner analogous to the way in which Gorodentsev, Kuleshov and Rudakov's t-stabilities generalise Bridgeland's theory of stability conditions. As an application of this notion, we use a complete classification of `generalised' co-slicings in the bounded derived category of the Kronecker algebra, $\sD^{b}(\bK Q)$, to obtain a classification of co-t-structures in $\sD^{b}(\bK Q)$. This is then used to compute the co-stability manifold of $\sD^{b}(\bK Q)$.
\end{abstract}

\dedicatory{This paper is dedicated to Hans-Bj\o rn Foxby on the occasion of his 65th birthday.}

\author{Peter J\o rgensen}
\address{School of Mathematics and Statistics,
Newcastle University, Newcastle upon Tyne NE1 7RU, United Kingdom}
\email{peter.jorgensen@ncl.ac.uk}
\urladdr{http://www.staff.ncl.ac.uk/peter.jorgensen}

\author{David Pauksztello}
\address{Institut f\"{u}r Algebra, Zahlentheorie und Diskrete
  Mathematik, Fa\-kul\-t\"{a}t f\"{u}r Ma\-the\-ma\-tik und Physik, Leibniz
  Universit\"{a}t Hannover, Welfengarten 1, 30167 Hannover, Germany}
\email{pauk@math.uni-hannover.de}


	
\maketitle

\section*{Introduction}\label{sec:intro}

Co-stability conditions were introduced in \cite{Co-stability} as a natural analogue of Bridgeland's stability conditions, which first appeared in \cite{Bridgeland}. Their introduction was motivated by two guiding principles: firstly, stability conditions may be considered as a continuous generalisation of bounded t-structures, and the theory of t-structures has a natural counterpart in the theory of co-t-structures, so a theory of co-stability conditions should also be a natural object of study. Secondly, and more concretely, there are mainstream examples of triangulated categories with no bounded t-structures (and hence no stability conditions) that do have bounded co-t-structures (and hence do have co-stability conditions); see \cite{HJY}, for instance.

However, computation of examples of stability conditions is an important and challenging problem. The difficulty of this problem led Gorodentsev, Kuleshov and Rudakov, see \cite{GKR}, to introduce a generalisation of Bridgeland's notion omitting the requirement for a stability function, which they called `t-stabilities'. One of the trade-offs they required was losing the complex manifold structure which was a central result of \cite{Bridgeland}. This approach yields a different perspective on the theory of stability conditions, and as an application Gorodentsev, Kuleshov and Rudakov were able to give a classification of the bounded t-structures in the bounded derived category of the Kronecker algebra.

It is natural to ask whether there is a corresponding theory of `co-t-stabilites' which generalises the co-stability conditions of \cite{Co-stability} in a manner analogous to that in which \cite{GKR} generalises \cite{Bridgeland}. Indeed, there is such a theory, and as an application we are able to classify all co-t-structures in the bounded derived category of the Kronecker algebra and then compute the co-stability manifold.

The results of this paper can be split up into two parts. The first part, consisting of sections \ref{sec:prelim}, \ref{sec:basic} and \ref{sec:almost} is formal homological algebra and works in any triangulated category $\sT$ (in section \ref{sec:almost} we also insist that $\sT$ is Krull-Schmidt). Aside from section \ref{sec:almost} this part is a formal analogue of the theory on t-stabilities appearing in \cite{GKR}.

The second part applies this theory to a concrete example, namely the Kronecker algebra. This well-understood example has the benefit of being the simplest example of interest both to representation theorists and geometers, by the Beilinson derived equivalence of the bounded derived categories of modules over the Kronecker algebra and coherent sheaves over the projective line; see \cite{Beilinson}.

Let us give a brief review of the main results of the paper. Let $\bK$ be an algebraically closed field and $Q$ be the Kronecker quiver
$\xymatrix{ 0  & 1 \ar@<-0.5ex>[l]\ar@<+0.5ex>[l]}$.
We denote by $\sD^{b}(\bK Q)$ the bounded derived category of finitely generated right $\bK Q$-modules. In sections \ref{sec:exceptional} and \ref{sec:co-slice_class} we obtain the following main theorem:

\begin{introtheorem}
All co-slicings (=`co-t-stabilities') of $\sD^b(\bK Q)$ are coarser than co-slicings obtained from exceptional collections of $\sD^b(\bK Q)$.
\end{introtheorem}

One of the main conclusions that can be drawn from this theorem is that all the co-t-structures induced by any co-slicing of $\sD^b(\bK Q)$ are induced by those co-slicings obtained from exceptional collections, so-called {\it exceptional co-slicings} (see section~\ref{sec:exceptional}). Note that here there is a crucial difference between co-slicings and the t-stabilities of \cite{GKR}: there is no co-slicing analogue of Gorodentsev, Kuleshov and Rudakov's `standard t-stabilities'.

The second main result of the paper shows that all co-t-structures in $\sD^b(\bK Q)$ are induced (see Definition~\ref{def:induced}) by exceptional co-slicings. Further, we can classify the co-t-structures in $\sD^b(\bK Q)$ as follows:

\begin{introtheorem}
Each exceptional pair of $\mod{\bK Q}$ gives rise to four families of non-trivial co-t-structures in $\sD^{b}(\bK Q)$: a family of bounded co-t-structures, a family of bounded below co-t-structures, a family of bounded above co-t-structures and a family of stable co-t-structures (= stable t-structures). Moreover, these are all the  co-t-structures in $\sD^b(\bK Q)$.
\end{introtheorem}

Note that, while all the co-t-structures induced by co-slicings are induced by exceptional co-slicings, showing that these are all the co-t-structures is non-trivial. The proof in our case relies on the fact that any partial silting subcategory of $\sD^b(\bK Q)$ is actually an almost complete silting subcategory. In section~\ref{sec:almost} we study co-slicings arising from almost complete silting subcategories; the properties of these co-slicings will be crucial to the proof that we have all co-t-structures in $\sD^b(\bK Q)$. As an aside, we believe this theory to be potentially quite rich and warrants further study.

As an application of the classification of co-t-structures in $\sD^b(\bK Q)$, we are able to compute the co-stability manifold of $\sD^b(\bK Q)$:

\begin{introtheorem}
The co-stability manifold of the bounded derived category of the Kronecker algebra, $\Costab{\sD^{b}(\bK Q)}$, is homeomorphic to $\bZ$ copies of $\bC^{2}$.
\end{introtheorem}

Note that this means we have an example of a triangulated category whose stability and co-stability manifolds are non-trivial. The stability manifold of $\sD^{b}(\bK Q)$ was computed to be $\bC^2$ by Okada in \cite{Okada}. Note that, the absence of an analogue of the `standard t-stabilities' of \cite{GKR} for co-slicings explains why the co-stability manifold decomposes into $\bZ$ components while the stability manifold does not.

\smallskip

\noindent{\bf Notation.} Categories will be denoted by Sans Serif letters, and objects in categories by lower case Roman letters in the abstract part of the paper (Sections \ref{sec:prelim}, \ref{sec:basic} and \ref{sec:almost}). In the remainder, we return to classic notation with modules and complexes being denoted by upper case Roman letters.

We always assume that subcategories are closed under isomorphisms;
that is, if $a$ is an object of a subcategory and $a \cong a'$ in the
ambient category, then $a'$ is also in the subcategory.  In
particular, each of the closure operations we apply to obtain
subcategories is to be understood as producing subcategories closed
under isomorphisms.

Note that in a triangulated category, we will sometimes write a distinguished triangle $\tri{x}{y}{z}$ as 
$$\xymatrix{ x\ar[r] & y,\ar[d] \\
 & z\ar@{~>}[ul]}$$
where the wiggly arrow $\xymatrix{z\ar@{~>}[r] & x}$ stands for the arrow of degree one $z\rightarrow \Sigma x$.

Recall also that in the distinguished triangle $\trilabels{x}{y}{z}{f}{g}{h}$, the object $z$ is called the \emph{cone} of the morphism $f$ and the object $x$ is called the \emph{co-cone} of the morphism $g$; both are unique up to a non-unique isomorphism.

Let $\sS$ be a subcategory of a triangulated category $\sT$. We say that $\sS$ is \emph{extension closed} if given a distinguished triangle $\tri{s'}{s}{s''}$ with $s'$ and $s''$ objects of $\sS$, then $s$ is also an object of $\sS$.


\section{Co-$t$-structures and generalised co-slicings}\label{sec:prelim}

Throughout Sections \ref{sec:prelim} and \ref{sec:basic}, $\sT$ will be an essentially small triangulated category with suspension functor $\Sigma\colon \sT \to \sT$.

In this section we introduce the basic notions of co-t-structures and generalised co-slicings and give some canonical examples. The following definition is due, independently, to \cite[Definition 2.4]{Co-t-structures} and \cite[Definition 1.1.1]{Bondarko}. In \cite{Bondarko}, co-t-structures are called `weight structures'.

\begin{definition} \label{def:co-t-structure}
A \emph{co-t-structure} in $\sT$ is a pair $(\sA,\sB)$ of full subcategories of $\sT$ which are closed under direct sums and direct summands satisfying the following conditions:
\begin{enumerate}
\item $\Sigma^{-1}\sA\subseteq \sA$ and $\Sigma \sB\subseteq \sB$;
\item $\Hom{\sT}{a}{b}=0$ for all $a\in\sA$ and $b\in\sB$;
\item For each object $t\in\sT$, there exists a distinguished triangle $$\tri{a}{t}{b}$$ with $a\in\sA$ and $b\in\sB$.
\end{enumerate}
The \emph{co-heart} of $(\sA,\sB)$ is $\sC:=\sA\cap\Sigma^{-1}\sB$. Note that $\sC$ is not necessarily an abelian category, and when it is, it is semisimple.
\end{definition}

\begin{definition} \label{def:bounded}
A co-t-structure $(\sA,\sB)$ is called \emph{bounded below} if $\sT=\bigcup_{i\in\bZ}\Sigma^{i}\sA$. It is called \emph{bounded above} if $\sT=\bigcup_{i\in\bZ}\Sigma^{i}\sB$. If it is both bounded above and below, it is called \emph{bounded}.
\end{definition}

The following definition is the mirror image of \cite[Definition 2.4]{GKR} and is a generalisation of \cite[Definition 5.3]{Co-stability} in the sense that a co-stability function on $\sA$ satisfying the split Harder-Narasimhan property induces the following split-stability data on $\sA$.

\begin{definition} \label{def:split-stability}
Let $\sA$ be an additive category. A pair $(\Phi,\{\cQ(\phi)\, |\, \phi\in\Phi\})$ consisting of 
\begin{enumerate}[(i)]
\item a linearly ordered set $\Phi$; and,
\item for each $\phi\in\Phi$, a full subcategory $\cQ(\phi)\subseteq\sA$ which is closed under direct summands and direct sums;
\end{enumerate}
is called \emph{split-stability data} if it satisfies the following conditions:
\begin{enumerate}
\item If $\phi_{1}<\phi_{2}$ then $\Hom{\sA}{\cQ(\phi_{1})}{\cQ(\phi_{2})}=0$;
\item Each object $0\neq a\in\sA$ decomposes as $a\cong a_{1}\oplus a_{2}\oplus\cdots\oplus a_{n}$ with $a_{i}\in\cQ(\phi_{i})$ and $\phi_{1}<\phi_{2}<\cdots<\phi_{n}$. Such a decomposition will be referred to as a \emph{split Harder-Narasimhan filtration}.
\end{enumerate}
\end{definition}

\begin{definition} \label{def:co-slicing}
A pair $(\Phi,\{\cQ(\phi)\, |\, \phi\in\Phi\})$ consisting of 
\begin{enumerate}[(i)]
\item a linearly ordered set $\Phi$; and,
\item for each $\phi\in\Phi$, a full subcategory $\cQ(\phi)\subseteq\sT$ which is closed under direct summands and extensions;
\end{enumerate}
is called a \emph{(generalised) co-slicing} of $\sT$ if it satisfies the following conditions:
\begin{enumerate}
\item There is an automorphism of ordered sets $\lambda:\Phi\rightarrow\Phi$ such that $\lambda(\phi)\geq\phi$ and $\Sigma\cQ(\phi)=\cQ(\lambda(\phi))$ for all $\phi\in\Phi$;
\item If $\phi_{1}<\phi_{2}$ then $\Hom{\sT}{\cQ(\phi_{1})}{\cQ(\phi_{2})}=0$;
\item Each object $0\neq t\in\sT$ sits in a tower
$$\xymatrix{0\cong t_{0}\ar[r] & t_{1}\ar[r]\ar[d] & t_{2}\ar[r]\ar[d] & \cdots\ar[r] & t_{n-1}\ar[r]\ar[d] & t_{n}\cong t\ar[d] \\
 & q_{1}\ar@{~>}[ul] & q_{2}\ar@{~>}[ul] & & q_{n-1}\ar@{~>}[ul] & q_{n}\ar@{~>}[ul] }$$
with $0\neq q_{i}\in\cQ(\phi_{i})$ and $\phi_{1}<\phi_{2}<\cdots < \phi_{n}$. 
\end{enumerate}
\end{definition}

We shall often refer to the tower of triangles in Definition \ref{def:co-slicing}$(3)$ as a \emph{Harder-Narasimhan (HN) filtration} of $t$. The full subcategories $\cQ(\phi)$ for $\phi\in\Phi$ will be called the \emph{co-slices} of the co-slicing. Following the terminology of \cite{Bridgeland} and \cite{GKR}, the objects of $\cQ(\phi)$ will be called \emph{semistable of phase $\phi$}.

We have the following canonical examples, which can be obtained from naturally occurring co-t-structures.

\begin{example} \label{ex:stable}
Suppose $(\sA,\sB)$ is a co-t-structure on a triangulated category $\sT$ satisfying $\Sigma\sA=\sA$ (and hence $\Sigma\sB =\sB$), i.e. it is a \emph{stable co-t-structure} (= stable t-structure). Define a co-slicing $(\Phi, \{\cQ(\phi)\, |\, \phi\in\Phi\})$ by
\begin{itemize}
\item $\Phi = \{0,1\}$ with $0< 1$;
\item $\cQ(0)=\sA$ and $\cQ(1)=\sB$;
\item $\lambda(0)=0$ and $\lambda(1)=1$.
\end{itemize}
\end{example}

\begin{example} \label{ex:bounded}
Suppose $(\sA,\sB)$ is a bounded co-t-structure on a triangulated category $\sT$. We can define a co-slicing $(\Phi, \{\cQ(\phi)\, |\, \phi\in\Phi\})$ on $\sT$ by
\begin{itemize}
\item $\Phi=\bZ$ with the standard linear ordering;
\item $\lambda(\phi)=\phi +1$ for each $\phi\in\bZ$;
\item $\cQ(\phi)=\Sigma^{\phi}\sC$, where $\sC=\sA\cap\Sigma^{-1}\sB$ is the co-heart of $(\sA,\sB)$.
\end{itemize}
\end{example}

\begin{example} \label{ex:co-slice}
Suppose $\cQ$ is a co-slicing of $\sT$ in the sense of \cite[Definition 3.1]{Co-stability}. Then the pair $(\cQ(\phi),\bR)$ is a generalised co-slicing with $\lambda:\bR\rightarrow\bR$ given by $\phi\mapsto\phi+1$.
\end{example}

This last example justifies the terminology `generalised co-slicing'. In what follows we will often omit the adjective `generalised' so that co-slicing refers to the concept defined in Definition \ref{def:co-slicing} rather than that defined in \cite[Definition 3.1]{Co-stability}. In a similar vein, the stability data of \cite{GKR} will be referred to as (generalised) slicings.

\begin{remark}
Example \ref{ex:bounded} is essentially \cite[Proposition 1.5.6]{Bondarko}. Note that in \cite[Proposition 5.3]{Bridgeland}, one has a bijective correspondence between bounded t-structures in a triangulated category $\sT$ with a stability function satisfying the HN property on its heart and stability conditions on $\sT$. However, the example of a stable (co)-t-structure shows that one does not need a bounded (co)-$t$-structure to obtain a (generalised) (co)-slicing.
\end{remark}

Before we leave the basic definitions, let us first recall the notion of silting. The terminology is originally due to \cite{KV}, however, the usage of the terminology in this article is drawn from \cite{AI}.

\begin{definitions} \label{def:silting}
Suppose $\sS$ is a subcategory of a triangulated category $\sT$.
\begin{enumerate}[leftmargin=*] 
\item The subcategory $\sS$ is called a \emph{partial silting subcategory} if $\Hom{\sT}{s}{\Sigma^{i}s'}=0$ for all $s,s'\in\sS$ and for all $i>0$. 
\item The subcategory $\sS$ is called a \emph{silting subcategory} if it is a partial silting subcategory and the smallest thick subcategory of $\sT$ containing $\sS$ is $\sT$, i.e. $\thick{\sT}{\sS}=\sT$.
\end{enumerate}
\end{definitions}


\section{Basic properties of co-slicings}\label{sec:basic}

In this section we collect some formal properties of co-slicings. Besides Proposition~\ref{prop:silting}, all the properties are analogous to the formal properties of slicings in \cite{GKR}. Thus we give very little detail in the proofs and include this section for the sake of completeness.

\begin{proposition} \label{prop:silting}
Suppose $(\Phi,\{\cQ(\phi)\,|\, \phi\in\Phi\})$ is a co-slicing of a triangulated category $\sT$. Then a co-slice $\cQ(\phi)$ for $\phi\in\Phi$ is either stable under suspensions (and desuspensions) or is partial silting.

In particular, if a co-slice $\cQ(\phi)$ of a co-slicing $(\Phi,\{\cQ(\phi)\,|\, \phi\in\Phi\})$ of $\sT$ is not stable under suspensions, then the extensions in $\cQ(\phi)$ are split.
\end{proposition}

\begin{proof}
Consider a co-slice $\cQ(\phi)$ of $(\Phi,\{\cQ(\phi)\,|\, \phi\in\Phi\})$. If $\Hom{\sT}{q}{\Sigma^{n}q'}=0$ for all $q,q'\in\cQ(\phi)$ and all $n>0$, then $\cQ(\phi)$ is partial silting. So suppose $\Hom{\sT}{q}{\Sigma^{n}q'}\neq 0$ for some $q,q'\in\cQ(\phi)$ and $n>0$. Then Definition \ref{def:co-slicing}$(2)$ implies that $\phi\geq\lambda^{n}(\phi)$. Definition \ref{def:co-slicing}$(1)$ gives a chain of inequalities:
$$\phi \geq \lambda^{n}(\phi) \geq \lambda^{n-1}(\phi) \geq \cdots \geq \lambda(\phi) \geq \phi.$$
Hence $\lambda(\phi)=\phi$ and $\Sigma\cQ(\phi)=\cQ(\phi)$. The second statement is now clear.
\end{proof}

The next proposition indicates how to combine a bounded co-t-structure with split-stability data on its co-heart.

\begin{proposition} \label{prop:bounded_co-slicing}
Let $\sT$ be a triangulated category with a bounded co-t-structure $(\sA,\sB)$. Suppose we are given split-stability data $(\Phi, \{\cR(\phi)\, |\, \phi\in\Phi\})$ on the co-heart $\sC$. We can define a co-slicing of $\sT$ by $(\Psi, \{\cQ(\psi)\, |\, \psi\in\Psi\})$ by
\begin{itemize}
\item $\Psi =\bZ\times\Phi$, with lexicographical ordering on $\Psi$;
\item $\cQ((n,\phi))=\Sigma^{n}\cR(\phi)$ for each $(n,\phi)\in\Psi$.
\end{itemize}
\end{proposition}

\begin{proof}
The proof is straightforward and analogous to that in \cite[Proposition 3.4]{GKR}: one simply checks the axioms in turn. The only not-so-obvious part is the existence of Harder-Narasimhan filtrations, so we give a sketch: Example \ref{ex:bounded} means that given $0\neq t\in\sT$, we can write down a tower
\begin{equation}
\xymatrix{0\cong t_{0}\ar[r] & t_{1}\ar[r]\ar[d] & t_{2}\ar[r]\ar[d] & \cdots\ar[r] & t_{n-1}\ar[r]\ar[d] & t_{n}\cong t\ar[d] \\
  & c_{1}\ar@{~>}[ul] & c_{2}\ar@{~>}[ul] & & c_{n-1}\ar@{~>}[ul] & c_{n}\ar@{~>}[ul] &   }
\label{eq:tower1}
\end{equation}
with $c_{j}\in\Sigma^{k_{j}}\sC$ and $k_{1}<k_{2}<\cdots <k_{n}$ (in $\bZ$). Since $(\Phi, \{\cR(\phi)\, |\, \phi\in\Phi\})$ is split-stability data on the co-heart $\sC$, the objects $\Sigma^{-k_{j}}c_{j}$ have split HN filtrations $$\Sigma^{-k_{j}}c_{j}\cong \Sigma^{-k_{j}}(r_{j}^{1}\oplus r_{j}^{2}\oplus\cdots\oplus r_{j}^{m_{j}})$$ with $\Sigma^{-k_{j}}r_{j}^{i}\in\cR(\phi_{j}^{i})$ and $\phi_{j}^{1}<\phi_{j}^{2}<\cdots<\phi_{j}^{m_{j}}$. Such a direct sum decomposition gives rise to a tower of split triangles:
\begin{equation}
\xymatrix{0\cong c_{j}^{0}\ar[r] & c_{j}^{1}\ar[r]\ar[d] & c_{j}^{2}\ar[r]\ar[d] & \cdots\ar[r] & c_{j}^{m_{j}-1}\ar[r]\ar[d] & c_{j}^{m_{j}}\cong c_{j}\ar[d] \\
 & r_{j}^{1}\ar@{~>}[ul]_-{0} & r_{j}^{2}\ar@{~>}[ul]_-{0} & & r_{j}^{m_{j}-1}\ar@{~>}[ul]_-{0} & r_{j}^{m_{j}}\ar@{~>}[ul]_-{0} &  }
\label{eq:split-tower}
\end{equation}
with $r_{j}^{i}\in\Sigma^{k_{j}}\cR(\phi_{j}^{i})=\cQ((k_{j},\phi_{j}^{i}))$ and $\phi_{j}^{1}<\phi_{j}^{2}<\cdots<\phi_{j}^{m_{j}}$. One can then glue together towers \eqref{eq:tower1} and \eqref{eq:split-tower} using \cite[Proposition 4.3]{GKR}.
\end{proof}

\begin{corollary}
Split-stability data $(\Phi, \{\cQ(\phi)\, |\, \phi\in\Phi\})$ on an additive category $\sA$ induces a co-slicing $(\bZ\times\Phi, \{\Sigma^{n}\cQ(\phi)\, |\, (n,\phi)\in\bZ\times\Phi\})$ on the bounded homotopy category $\sK^{b}(\sA)$ of $\sA$.
\end{corollary}

\begin{lemma} \label{lem:family}
Suppose $(\Phi, \{\cQ(\phi)\, |\, \phi\in\Phi\})$ is a co-slicing of a triangulated category $\sT$ and that the linearly ordered set $\Phi$ can be written as a disjoint union $\Phi=\Phi_{-}\cup\Phi_{+}$ with the property that $$\phi_{-}\in\Phi_{-} \mbox{ and } \phi_{+}\in\Phi_{+} \implies \phi_{-} < \phi_{+}.$$ Then the following pair of subcategories,
$$\sA = (\bigcup_{\phi\in\Phi_{-}} \cQ(\phi))^{+} \quad {\rm and } \quad \sB  =  (\bigcup_{\phi\in\Phi_{+}} \cQ(\phi))^{+},$$
where $(-)^{+}$ indicates closure under extensions and direct summands, defines a co-t-structure in $\sT$.
\end{lemma}

\begin{proof}
The subcategories $\sA$ and $\sB$ are closed under direct sums and direct summands by definition. The inclusion $\Sigma^{-1}\sA\subseteq\sA$ follows from the fact $\Sigma^{-1}\cQ(\phi)=\cQ(\lambda^{-1}(\phi))$ and $\lambda^{-1}(\phi)\leq\phi$. Similarly for $\Sigma\sB\subseteq\sB$. The orthogonality $\Hom{\sT}{\sA}{\sB}=0$ follows from Definition \ref{def:co-slicing}(2), and the approximation triangle comes from the tower in Definition \ref{def:co-slicing}(3).
\end{proof}

This lemma, and the corollary below, correspond to Bridgeland's observation in \cite{Bridgeland} that a slicing induces a family of `compatible' t-structures. With this in mind, we make the following definition.

\begin{definition} \label{def:induced}
We say that the co-t-structure $(\sA,\sB)$ obtained in Lemma \ref{lem:family} is \emph{induced} by $(\Phi,\{\cQ(\phi)\,|\,\phi\in\Phi\})$.
\end{definition}

We have the following specialisation of Lemma \ref{lem:family}.

\begin{corollary} \label{cor:family}
Suppose $(\Phi, \{\cQ(\phi)\, |\, \phi\in\Phi\})$ is a co-slicing of a triangulated category $\sT$. Then each $\phi_{0}\in\Phi$ determines a co-t-structure defined by the subcategories:
$$\sA  =  (\bigcup_{\phi<\phi_{0}} \cQ(\phi))^{+} \quad {\rm and } \quad \sB  = (\bigcup_{\phi\geq\phi_{0}} \cQ(\phi))^{+}.$$
\end{corollary}

The following easy lemma is the analogue of the remark after \cite[Corollary 5.2]{GKR}. In \cite{GKR}, no proof is given, so we give a proof for the convenience of the reader.

\begin{lemma} \label{lem:bounded}
Let $(\Phi, \{\cQ(\phi)\, |\, \phi\in\Phi\})$ be a co-slicing of $\sT$ satisfying the hypotheses of Lemma \ref{lem:family}, and suppose that $(\sA,\sB)$ is the co-t-structure induced by the decomposition $\Phi=\Phi_{-}\cup\Phi_{+}$. Then:
\begin{enumerate}[label=(\roman*),leftmargin=*]
\item $(\sA,\sB)$ is bounded below if and only if $\Phi=\bigcup_{n\geq 0}\lambda^{n}(\Phi_{-})$.
\item $(\sA,\sB)$ is bounded above if and only if $\Phi=\bigcup_{n\geq 0}\lambda^{-n}(\Phi_{+})$.
\item $(\sA,\sB)$ is bounded if and only if $\Phi=\bigcup_{n\geq 0}\lambda^{n}(\Phi_{-})=\bigcup_{n\geq 0}\lambda^{-n}(\Phi_{+})$.
\end{enumerate}
\end{lemma}

\begin{proof}
We just prove (i), (ii) is proved analogously and (iii) is proved by combining (i) and (ii).

Suppose $(\sA,\sB)$ is bounded below, i.e. $\sT=\bigcup_{i\in\bZ}\Sigma^{i}\sA$. Now consider $\phi\in\Phi_{+}$. Then we have $\cQ(\phi)\subseteq\sB$. Now given $q\in\cQ(\phi)$ there exists $n>0$ such that $q\in\Sigma^{n}\sA$ since $(\sA,\sB)$ is bounded below, and hence, $\Sigma^{-n}q\in\sA$. But $\Sigma^{-n}q$ is semistable with respect to $(\Phi, \{\cQ(\phi)\, |\, \phi\in\Phi\})$ so that there exists $\phi_{-}\in\Phi_{-}$ such that $\Sigma^{-n}q\in\cQ(\phi_{-})$. It follows that $q\in\Sigma^{n}\cQ(\phi_{-})=\cQ(\lambda^{n}(\phi_{-}))$, in which case $\cQ(\phi)=\cQ(\lambda^{n}(\phi_{-}))$ and we obtain $\phi=\lambda^{n}(\phi_{-})$. Thus we have $\Phi_{+}\subseteq\bigcup_{n>0}\lambda^{n}(\Phi_{-})$ and so it follows that $\Phi=\bigcup_{n\geq 0}\lambda^{n}(\Phi_{-})$.

Conversely, suppose $\Phi=\bigcup_{n\geq 0}\lambda^{n}(\Phi_{-})$ and suppose $0\neq t\in\sT$. Then, since $(\Phi, \{\cQ(\phi)\, |\, \phi\in\Phi\})$ is a co-slicing we obtain an HN filtration:
$$\xymatrix{0\cong t_{0}\ar[r] & t_{1}\ar[r]\ar[d] & t_{2}\ar[r]\ar[d] & \cdots\ar[r] & t_{n-1}\ar[r]\ar[d] & t_{n}\cong t\ar[d] \\
  & q_{1}\ar@{~>}[ul] & q_{2}\ar@{~>}[ul] & & q_{n-1}\ar@{~>}[ul] & q_{n}\ar@{~>}[ul]   & }$$
with $0\neq q_{i}\in\cQ(\phi_{i})$ and $\phi_{1}<\phi_{2}<\cdots < \phi_{n}$.  If all $\phi_{i}\in\Phi_{-}$ then $t\in\sA$ and we are done. If all $\phi_{i}\in\Phi_{+}$ then there exist $n_{i}>0$ so that $\phi_{i}=\lambda^{n_{i}}(\psi_{i})$ with $\psi_{i}\in\Phi_{-}$. Taking $n=\max\{n_{i}\}$ gives $q_{i}\in\Sigma^{n}\sA$ for each $i$. Then extension closure of $\Sigma^{n}\sA$ gives that $t\in\Sigma^{n}\sA$. If we have some $\phi_{i}$ in $\Phi_{+}$ and some in $\Phi_{-}$, then a similar argument also works. Hence we obtain $\sT=\bigcup_{n\in\bZ}\Sigma^{n}\sA$, as required.
\end{proof}

Like in the case of slicings in \cite{GKR}, certain co-slicings induce the same co-t-structures as others. In a manner analogous to the partial order on the set of slicings on a triangulated category given in \cite{GKR}, we define a partial order on the set of co-slicings. We first need another technical definition.

\begin{definition} \label{def:essential}
A co-slicing $(\Phi,\{\cQ(\phi)\,|\,\phi\in\Phi\})$ will be called \emph{essential} if each of its co-slices $\cQ(\phi)$ is non-zero.
\end{definition}

\begin{definition}[Definition 5.3 in \cite{GKR}] \label{def:finer}
Let $(\Phi,\{\cQ(\phi)\,|\,\phi\in\Phi\})$ and $(\Psi,\{\cR(\psi)\,|\,\psi\in\Psi\})$ be co-slicings of a triangulated category $\sT$. Denote the corresponding automorphisms of linearly ordered sets by $\lambda:\Phi\rightarrow\Phi$ and $\mu:\Psi\rightarrow\Psi$, respectively. We say that $\Phi$ is \emph{finer} than $\Psi$ (or $\Psi$ is \emph{coarser} than $\Phi$) if there exists a surjection $r:\Phi\rightarrow\Psi$ such that
\begin{enumerate}
\item $r\circ\lambda = \mu\circ r$;
\item If $\phi_{1}<\phi_{2}$ then $r(\phi_{1})\leq r(\phi_{2})$;
\item For all $\psi\in\Psi$ we have $\cR(\psi)=(\bigcup_{\phi\in r^{-1}(\psi)}\cQ(\phi))^{+}$.
\end{enumerate}
\end{definition}

\begin{remark} \label{rem:essential}
Given a co-slicing of $\sT$, one can always find a finer co-slicing by adding zero co-slices. However, this finer co-slicing will contain exactly the same information as the one with which we started.
\end{remark}

The following lemma says that when determining which co-t-structures are induced by co-slicings, we need only consider the minimal elements with respect to the partial order defined above; we shall refer to such co-slicings as \emph{minimal co-slicings}.

\begin{lemma} \label{lem:compare}
If the co-slicing $(\Psi,\{\cR(\psi)\,|\,\psi\in\Psi\})$ is coarser than $(\Phi,\{\cQ(\phi)\,|\,\phi\in\Phi\})$, then a co-t-structure in $\sT$ induced by the co-slicing $(\Psi,\{\cR(\psi)\,|\,\psi\in\Psi\})$ is also induced by the co-slicing $(\Phi,\{\cQ(\phi)\,|\,\phi\in\Phi\})$.
\end{lemma}

\begin{proof}
By Lemma \ref{lem:family}, there is a co-t-structure in $\sT$ defined by
$$
\sA_{\Psi} := (\bigcup_{\psi\in\Psi_{-}}\cR(\psi))^{+} \quad {\rm and } \quad
\sB_{\Psi}  :=  (\bigcup_{\psi\in\Psi_{+}}\cR(\psi))^{+}.
$$
Since $(\Psi,\{\cR(\psi)\,|\,\psi\in\Psi\})$ is coarser than $(\Phi,\{\cQ(\phi)\,|\,\phi\in\Phi\})$, there exists a surjection $r:\Phi\rightarrow\Psi$ satisfying the conditions of Definition \ref{def:finer}. Let $\Phi_{\psi}=r^{-1}(\psi)\subset\Phi$.
Now, setting $\Phi_{-} = \{\phi\in\Phi\,|\,\phi\in\Phi_{\psi},\,\psi\in\Psi_{-}\}$ and $\Phi_{+}  = \{\phi\in\Phi\,|\,\phi\in\Phi_{\psi},\,\psi\in\Psi_{+}\}$ and writing $\Phi=\Phi_{-}\cup\Phi_{+}$, Lemma \ref{lem:family} gives us a co-t-structure defined by:
$$\sA_{\Phi} := (\bigcup_{\phi\in\Phi_{-}}\cQ(\phi))^{+} \quad {\rm and} \quad
\sB_{\Phi} :=  (\bigcup_{\phi\in\Phi_{+}}\cQ(\phi))^{+}.$$
Now one simply observes that
$\sA_{\Psi}=(\bigcup_{\psi\in\Psi_{-}}\bigcup_{\phi\in\Phi_{\psi}}\cQ(\phi))^{+} =(\bigcup_{\phi\in\Phi_{-}}\cQ(\phi))^{+}=\sA_{\Phi}.$
Similarly, $\sB_{\Psi}=\sB_{\Phi}$.
\end{proof}

\begin{remark} \label{rem:criterion}
Note that for an essential co-slicing $(\Phi,\{\cQ(\phi)\,|\,\phi\in\Phi\})$ of $\sT$, if for all $\phi\in\Phi$ and all non-zero $x,y\in\cQ(\phi)$ we have $\Hom{\sT}{x}{y}\neq 0$ and $\Hom{\sT}{y}{x}\neq 0$ then $(\Phi,\{\cQ(\phi)\,|\,\phi\in\Phi\})$ is minimal amongst the essential co-slicings.
\end{remark}


\section{Co-slicings coming from almost silting subcategories}\label{sec:almost}

In this section we study co-slicings arising from almost silting subcategories (Definition~\ref{def:almost-silting}). In the case of the bounded derived category of the Kronecker algebra, a partial silting subcategory that is not silting is almost silting. Thus, the results of this section will be especially applicable to the case of the Kronecker algebra, and indeed, will be crucial to our proof that we have all the co-t-structures for the Kronecker algebra in Theorem~\ref{thm:all}. 
However, the theory works for Krull-Schmidt triangulated categories, so the results are stated and proved in this level of generality. Moreover, we believe that this theory could hold independent interest.

\begin{definition} \label{def:almost-silting}
Let $\sT$ be a Krull-Schmidt triangulated category. A partial silting subcategory $\sS$ of $\sT$ is called an \emph{almost silting subcategory} if there exists an indecomposable object $s$ of $\sT$ such that $\add {\sS\cup\{s\}}$ is a silting subcategory of $\sT$.
\end{definition}

\begin{proposition} \label{prop:silting-bounded}
Let $\sT$ be a Krull-Schmidt triangulated category. Suppose $(\sA,\sB)$ is a co-t-structure in $\sT$ whose co-heart $\sC$ is an almost silting subcategory. Suppose $s \in \ind \sT$, is a complement for $\sC$, i.e. $\sS = \add{\sC \cup \{s\}}$ is a silting subcategory. Then:
\begin{enumerate}[label=(\roman*),leftmargin=*]
\item If $\Hom{\sT}{\sC}{\Sigma^n s} = 0$ for all $n \leq 0$ then the co-t-structure $(\sA,\sB)$ is either bounded above or below.
Moreover, $\Sigma^k s \in \sB$ for some $k \in \bZ$ if and only if $\Sigma^k s \in \sB$ for all $k \in \bZ$; in this case $(\sA,\sB)$ is bounded above.
\item If $\Hom{\sT}{s}{\Sigma^n \sC} = 0$ for all $n \leq 0$ then the co-t-structure $(\sA,\sB)$ is either  bounded above or below.
Moreover, $\Sigma^k s \in \sA$ for some $k \in \bZ$ if and only if $\Sigma^k s \in \sA$ for all $k \in \bZ$; in this case $(\sA,\sB)$ is bounded below.
\end{enumerate}
\end{proposition}

\begin{proof}
We prove statement (i); the proof for (ii) is analogous.

Since $\sS$ is a silting subcategory, it is the co-heart of a bounded co-t-structure $(\sA_{\sS},\sB_{\sS})$  by \cite[Corollary 5.8]{MSSS}. Hence, by \cite[Proposition 1.5.6]{Bondarko}, each object $0\neq t$ sits in a tower of distinguished triangles:
\begin{equation}
\label{eq:canonical}
\xymatrix{
0 \cong t_{0}\ar[r] & t_{1}\ar[r]\ar[d]                             & t_{2}\ar[r]\ar[d]                            & \cdots\ar[r] & t_{n-1}\ar[r] & t_{n}\cong t \ar[d] \\
                    & \Sigma^{i_{1}}( c_1 \oplus s_1 ) \ar@{~>}[ul] & \Sigma^{i_{2}}( c_2 \oplus s_2 )\ar@{~>}[ul] &              &               & \Sigma^{i_{n}}( c_n \oplus s_n ) \ar@{~>}[ul] 
}
\end{equation}
with $i_1 < i_2 < \cdots < i_n$ and $c_i\in \sC$ and $s_i \in \add{s}$. 
Since $\sS$ is a silting set and  $\Hom{\sT}{\sC}{\Sigma^n s} = 0$ for all $n \leq 0$, one can apply \cite[Lemmas 7.1 and 7.2]{Co-stability} iteratively to the tower above to get the following tower of distinguished triangles:
\begin{equation}
\label{eq:mixed-up}
\xymatrix{
0 \cong t'_0\ar[r] & t'_1\ar[r]\ar[d]                             & t'_2\ar[r]\ar[d]                            & \cdots\ar[r] & t'_{n-1}\ar[r] & t'_n\ar[r]\ar[d] & t_{n+1}'\cong t \ar[d] \\
                    & \Sigma^{i_{1}} c_1 \ar@{~>}[ul] & \Sigma^{i_{2}} c_2 \ar@{~>}[ul] &              &               & \Sigma^{i_{n}} c_n \ar@{~>}[ul] & s'\ar@{~>}[ul]
}
\end{equation}
with $s'=\Sigma^{i_1} s_1 \oplus \cdots \oplus \Sigma^{i_n} s_n$.

If $\Sigma^k s \in \sA$ for some $k\in \bZ$, then $(\sA,\sB)$ is bounded below and there is nothing to prove. So assume $\Sigma^k s \notin \sA$ for any $k \in \bZ$. Let $0\neq a\in \sA$ and consider the tower \eqref{eq:mixed-up} with $t=a$. The first $n$ triangles of tower \eqref{eq:mixed-up} is thus an HN filtration of $t'_n$ and thus we find that $t'_n \in \Sigma^{i_n} \sA$. Using the triangle $\tri{t'_n}{a}{s'}$, we see that $s' \in \Sigma^j \sA$ where $j=\max \{0, i_n +1\}$, which implies $s' \cong 0$. Hence, $a\cong t'_n \in \thick{\sT}{\sC}$, giving $\sA \subseteq \thick{\sT}{\sC}$. Now \cite[Remark 4.6(b)]{MSSS} implies that $(\sA,\sB)$ is bounded above, proving the first claim.
(Note \cite[Theorem 4.5(b)]{MSSS} says $\thick{\sT}{\sC} = \sC^{\sim}$, and \cite[Remark 4.6(b)]{MSSS} gives the inclusion $\sA \subseteq \sC^{\sim}$; see \cite[Definition 2.2]{MSSS} for a precise description of this notation.)

For the second claim, one direction is clear. For the other, assume that $\Sigma^k s \in \sB$ for some $k\in\bZ$. We have $\thick{\sT}{s}\subseteq \thick{\sT}{\sC}^{\perp}$ by assumption, and for each $0\neq t \in \sT$, the tower \eqref{eq:mixed-up} gives a decomposition $\tri{c'}{t}{s'}$ with $c' \in \thick{\sT}{\sC}$ and $s'\in \thick{\sT}{s}$. Thus, $(\thick{\sT}{\sC},\thick{\sT}{s})$ is a stable t-structure in $\sT$. Hence, $\thick{\sT}{s}= \thick{\sT}{\sC}^{\perp}$. Now, since $(\sA,\sB)$ is bounded above,  \cite[Remark 4.6(b)]{MSSS} implies that $\sA \subseteq \thick{\sT}{\sC}$, in which case $\thick{\sT}{s}\subseteq \sB$, as required.
\end{proof}

\begin{corollary} \label{cor:silting-bounded}
Let $(\sA,\sB)$ be a co-t-structure in $\sT$ satisfying the hypotheses of Proposition~\ref{prop:silting-bounded}. Suppose further that $\Hom{\sT}{\sC}{\Sigma^n s} = 0$ for all $n \leq 0$ and $(\sA,\sB)$ is bounded above. Then the following data define a co-slicing $(\Phi, \{\cQ(\phi)\,|\, \phi\in\Phi\})$ of $\sT$:
\begin{itemize}
\item $\Phi=\bZ \cup \{\infty\}$ with the standard linear order on $\bZ$ and $n<\infty$ for all $n\in\bZ$;
\item $\lambda(n) = n+1$ for $n\in \bZ$ and $\lambda(\infty)=\infty$;
\item $\cQ(n) = \Sigma^n \sC$ for $n\in \bZ$ and $\cQ(\infty)=\bigcap_{i \in \bZ} \Sigma^i \sB$.
\end{itemize}
\end{corollary}

\begin{proof}
The co-t-structure $(\sA,\sB)$ is bounded above, so $s\in\Sigma^k \sB$ for some $k\in \bZ$. Proposition \ref{prop:silting-bounded}(i) now implies that $\Sigma^k s \in \bigcap_{n\in\bZ} \Sigma^n \sB$ for all $k\in \bZ$. 

To obtain the HN filtration of $0\neq t\in\sT$, one can write down a canonical HN filtration with respect to the bounded co-t-structure $(\sA_{\sS},\sB_{\sS})$ as in \eqref{eq:canonical} and then use \cite[Lemmas 7.1 and 7.2]{Co-stability} to obtain the tower \eqref{eq:mixed-up}. This tower is now an HN filtration for $t$ with respect to the co-slicing described above.
\end{proof}

\begin{corollary} \label{cor:silting-bounded-induced}
Let $(\sA,\sB)$ and $(\Phi, \{\cQ(\phi) \, | \, \phi\in\Phi\})$ be as in Corollary~\ref{cor:silting-bounded}. Let $(\sA_n,\sB_n)$ be the co-t-structure of Corollary~\ref{cor:family} induced by some $n \in \bZ \subset \Phi$. Then $(\sA_n,\sB_n)=(\Sigma^{n-1} \sA,\Sigma^{n-1} \sB)$.
\end{corollary}

\begin{proof}
Suppose $n \in \bZ\subset \Phi$. Then $(\sA_n,\sB_n)$ is given by
\[
\sA_n = (\bigcup_{i< n} \cQ(i))^{+} \quad \text{and} \quad \sB_n = ((\bigcup_{i \geq n} \cQ(i))\cup \cQ(\infty))^{+}.
\]
Now since $\cQ(i) \subset \Sigma^i \sC \subset \Sigma^{n-1} \sA$ for all $i < n$, and $\sA$ is closed under extensions and direct summands, it follows that $\sA_n \subset \Sigma^{n-1}\sA$. Similarly, $\sB_n \subset \Sigma^{n-1}\sB$. Thus, by \cite[Lemma 1.3.8]{Bondarko}, it follows that $\sA_n = \Sigma^{n-1}\sA$ and $\sB_n = \Sigma^{n-1}\sB$.
\end{proof}

Statements analogous to Corollaries~\ref{cor:silting-bounded} and \ref{cor:silting-bounded-induced} hold when $\Hom{\sT}{s}{\Sigma^{n}\sC} = 0$ for $n\leq 0$ and $(\sA,\sB)$ is bounded below. Note that in Corollary~\ref{cor:silting-bounded-induced}, if one were to take $n=\infty$, the co-t-structure $(\sA_{\infty},\sB_{\infty})$ is a stable co-t-structure.


\section{Recapitulation on the Kronecker algebra}
\label{sec:recap}

In this section we provide a brief recapitulation of the information we need regarding the structure of the bounded derived category of the Kronecker algebra. Background information on Auslander-Reiten theory can be found in \cite{ASS} and \cite{ARS}. Specific information on the module category of the Kronecker algebra can be found in \cite[Section VIII.7]{ARS} and \cite[Section XI.4]{SS}. The standard reference for the derived category of a finite-dimensional hereditary algebra is \cite{Happel}.

Let $\bK$ be an algebraically closed field, and $Q$ be the Kronecker quiver. The path algebra $\bK Q$ is called the \emph{Kronecker algebra}. We shall denote by $\mod{\bK Q}$ the category of finite dimensional left $\bK Q$-modules. The structure of the Auslander-Reiten quiver of $\mod{\bK Q}$ is given below:

\

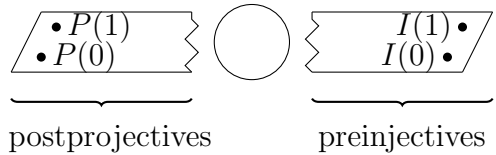
\begin{center}
\begin{tikzpicture}
\draw (0,0) -- (2.4,0);
\draw (0.4,0.8) -- (2.4,0.8);
\draw [decorate,decoration=zigzag] (2.4,0) -- (2.4, 0.8);
\draw (0,0) -- (0.4, 0.8);

\fill (0.6,0.6) circle (0.6mm) node[right] {$P(1)$};
\fill (0.4,0.2) circle (0.6mm) node[right] {$P(0)$};

\draw [thick,decorate,decoration={brace,mirror,raise=10pt}] (0,0) -- (2.4,0)
		node[pos=0.55,anchor=north,yshift=-0.55cm] {postprojectives};

\draw (3.2,0.4) circle (5mm);

\draw (4.0,0) -- (6,0);
\draw (4,0.8) -- (6.4,0.8);
\draw [decorate,decoration=zigzag] (4.0,0) -- (4.0, 0.8);
\draw (6,0) -- (6.4, 0.8);

\draw [thick,decorate,decoration={brace,mirror,raise=10pt}] (4,0) -- (6.4,0)
		node[pos=0.5,anchor=north,yshift=-0.55cm] {preinjectives};

\fill (6,0.6) circle (0.6mm) node[left] {$I(1)$};
\fill (5.8,0.2) circle (0.6mm) node[left] {$I(0)$};
\end{tikzpicture}
\end{center}

\noindent where $P(0)$ and $P(1)$ are the indecomposable projective modules and $I(0)$ and $I(1)$ are the indecomposable injective modules. The central circle represents the regular components, which consists of a $\bP^{1}$-indexed family of stable homogeneous tubes.

A result of Happel implies that each object of the derived category $\sD^{b}(\bK Q)$ of $\mod{\bK Q}$ decomposes as a direct sum of its cohomology; for details see \cite[Lemma I.5.2 and Corollary I.5.3]{Happel}. Hence the AR quiver of the derived category can be sketched as:

\

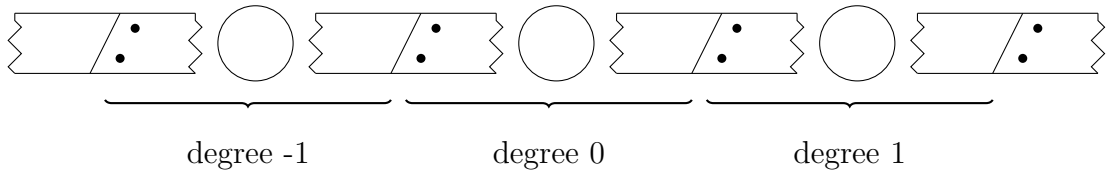
\begin{center}
\begin{tikzpicture}
\draw (0,0) -- (2.4,0);
\draw (0,0.8) -- (2.4,0.8);
\draw [decorate,decoration=zigzag] (0,0) -- (0, 0.8);
\draw [decorate,decoration=zigzag] (2.4,0) -- (2.4, 0.8);
\draw (1,0) -- (1.4, 0.8);

\fill (1.6,0.6) circle (0.6mm);
\fill (1.4,0.2) circle (0.6mm);

\draw (3.2,0.4) circle (5mm);

\draw[thick,decorate,decoration={brace,mirror,raise=10pt}] (1.2,0) -- (5,0) node[pos=0.5,anchor=north,yshift=-0.75cm] {degree -1};

\draw (4.0,0) -- (6.4,0);
\draw (4,0.8) -- (6.4,0.8);
\draw [decorate,decoration=zigzag] (4.0,0) -- (4.0, 0.8);
\draw [decorate,decoration=zigzag] (6.4,0) -- (6.4, 0.8);
\draw (5,0) -- (5.4, 0.8);

\fill (5.6,0.6) circle (0.6mm);
\fill (5.4,0.2) circle (0.6mm);

\draw[thick,decorate,decoration={brace,mirror,raise=10pt}] (5.2,0) -- (9,0) node[pos=0.5,anchor=north,yshift=-0.75cm] {degree 0};

\draw (7.2,0.4) circle (5mm);

\draw (8.0,0) -- (10.4,0);
\draw (8.0,0.8) -- (10.4,0.8);
\draw [decorate,decoration=zigzag] (8.0,0) -- (8.0, 0.8);
\draw [decorate,decoration=zigzag] (10.4,0) -- (10.4, 0.8);
\draw (9,0) -- (9.4, 0.8);

\fill (9.6,0.6) circle (0.6mm);
\fill (9.4,0.2) circle (0.6mm);

\draw (11.2,0.4) circle (5mm);

\draw[thick,decorate,decoration={brace,mirror,raise=10pt}] (9.2,0) -- (13,0) node[pos=0.5,anchor=north,yshift=-0.75cm] {degree 1};

\draw (12.0,0) -- (14.4,0);
\draw (12.0,0.8) -- (14.4,0.8);
\draw [decorate,decoration=zigzag] (12.0,0) -- (12.0, 0.8);
\draw [decorate,decoration=zigzag] (14.4,0) -- (14.4, 0.8);
\draw (13,0) -- (13.4, 0.8);

\fill (13.6,0.6) circle (0.6mm);
\fill (13.4,0.2) circle (0.6mm);
\end{tikzpicture}
\end{center}

\noindent In each degree we have indicated the indecomposable projective modules. Sometimes it will be useful to use the following notation.

\begin{notation} \label{not:non-regulars}
We set 
$$P_{t} =
\left\{
\begin{array}{ll}
\tau^{-\frac{t}{2}}P(0) & \mbox{ if $t$ is even,} \\
\tau^{-\frac{t-1}{2}}P(1) & \mbox{ if $t$ is odd,}
\end{array} \right.
\quad \mbox{and}\quad
I_{t} =
\left\{
\begin{array}{ll}
\tau^{\frac{t}{2}}I(1) & \mbox{ if $t$ is even,} \\
\tau^{\frac{t-1}{2}}I(0) & \mbox{ if $t$ is odd,}
\end{array} \right.
$$
where $\tau$ is the AR translation of $\sD^{b}(\bK Q)$. Sometimes the following notation will be useful:
$$
N_{i} := \left\{
\begin{array}{ll}
\Sigma^{-1}I_{-i} & {\rm for }\ i\leq 0, \\
P_{i-1} & {\rm for }\ i>0.
\end{array}
\right.
$$
This corresponds to considering the `geometric heart' consisting of coherent sheaves on $\bP^{1}$. Indeed, there is a bijection $N_{i}\mapsto\cO(i)$, where $\cO(i)$ is the line bundle of degree $i$. The regular modules then correspond to torsion sheaves. This is formalised in the Beilinson equivalence \cite{Beilinson}.\end{notation}

The indecomposable objects of $\sD^{b}(\bK Q)$ are simply the stalk complexes of indecomposable $\bK Q$-modules, i.e. indecomposable $\bK Q$-modules considered as complexes concentrated in one degree. We shall use the following terminology: an \emph{indecomposable regular object} is simply the stalk complex of an indecomposable regular module; analogously for an \emph{indecomposable non-regular object}.

We now briefly recall some standard facts about $\sD^{b}(\bK Q)$ that will be useful later.

\begin{facts} \label{facts}
Let $R$ and $R'$ be indecomposable regular objects sitting in different tubes of $\sD^{b}(\bK Q)$ and $N_{i}$ be as above. Then
\begin{enumerate}
\item $\Hom{\sD^{b}(\bK Q)}{R}{\Sigma R}=\bK$.
\item $\Hom{\sD^{b}(\bK Q)}{R}{R'}=\Hom{\sD^{b}(\bK Q)}{R}{\Sigma R'}=0$.
\item $\Hom{\sD^{b}(\bK Q)}{N_{i}}{\Sigma N_{i-1}}=0$.
\end{enumerate}
If $P$ is an indecomposable postprojective object and $I$ is an indecomposable preinjective object sitting in the same degree as $R$, then:
\begin{enumerate}
\item[(4)] There are non-zero maps $\Sigma^{-1}R\rightarrow P$ and $P\rightarrow R$.
\item[(5)] There are non-zero maps $R\rightarrow I$ and $I\rightarrow \Sigma R$.
\end{enumerate}
\end{facts}

The following distinguished triangles are well-known and in geometry they are often know as the `Euler sequence'. They can be computed explicitly by taking projective resolutions of the indecomposable modules.

\begin{triangles} \label{triangles}
Denote by $R_{x,d}$ the indecomposable regular module of regular length $d\geq 1$ indexed by the point $x\in\bP^{1}$. The standard triangles are given below:
$$
\left.\begin{array}{l}
P_{1}^{\oplus d} \rightarrow R_{x,d}\rightarrow (\Sigma P_{0})^{\oplus d} \rightarrow \Sigma P_{1}^{\oplus d}, \\ 
P_{1}^{\oplus t} \rightarrow P_{t}\rightarrow (\Sigma P_{0})^{\oplus t-1} \rightarrow \Sigma P_{1}^{\oplus t}, \\
P_{1}^{\oplus s+1} \rightarrow I_{s}\rightarrow (\Sigma P_{0})^{\oplus s+2} \rightarrow \Sigma P_{1}^{\oplus s+1}, 
\end{array}\right.
$$
where $t\geq 2$ and $s\geq 0$.
\end{triangles}

\begin{remark}
Note that the cones and co-cones of morphisms in $\sD^{b}(\bK Q)$ are easily computed. For example, if $X$ and $Y$ are objects concentrated in the same degree and $f:X\rightarrow Y$ is a morphism, then, since $\bK Q$ is hereditary the cone of $f$ is the object $\coker f\oplus\Sigma(\ker f)$. This, together with the fact there are no morphisms of degree higher than one, can be used to compute any cone or co-cone.
\end{remark}


\section{Exceptional co-slicings of $\sD^{b}(\bK Q)$}\label{sec:exceptional}

The category $\sD^{b}(\bK Q)$ is classically generated by the exceptional pairs $\{N_{n},N_{n+1}\}$ for $n\in\bZ$ and their suspensions. It is well known that these are all the exceptional pairs in $\sD^b(\bK Q)$. Without loss of generality, we only explain the construction of exceptional co-slicings using the exceptional pair $\{P_{0},P_{1}\}$; the others are obtained analogously. The construction is based on the observation of Proposition \ref{prop:bounded_co-slicing} and proceeds along the lines of the corresponding construction in \cite{GKR}, so we give few details. 

We would like to define the co-slicing $(\cE,\{\cQ(\epsilon)\,|\,\epsilon\in\cE\})$ of $\sD^{b}(\bK Q)$ by:
\begin{itemize}
\item $\cE=\bZ\times\{0,1\}$;
\item $\cQ((n,0))=\add{\Sigma^{n}P_{0}}$ and $\cQ((n,1))=\add{\Sigma^{n}P_{1}}$.
\end{itemize}
We start with the following lemma.

\begin{lemma} \label{lem:linear_order}
The only admissible linear orderings of $\cE=\bZ\times\{0,1\}$ with an automorphism $\eta:\cE\rightarrow\cE$ which is compatible with the suspension applied to the co-slices above, i.e. satisfying $\eta(n,-)=(n+1,-)$, are given by:
\begin{enumerate}
\item We have $(n,i)<(m,i)$ for $i\in\{0,1\}$ and $n<m$, and $(n,1)<(m,0)$ for any $n,m\in\bZ$.
\item There exists $p\in\bN$ such that
$$\cdots<(n+p-1,1)<(n,0)<(n+p,1)<(n+1,0)<(n+p+1,1)<(n+2,0)<\cdots.$$
\end{enumerate}
We shall denote the set $\cE$ together with one of the possible linear orderings above by $\cE_{p}$ with $p\in\bN\cup\{\infty\}$, where if $p\in\bN$ the linear ordering corresponds to one of those occurring in (ii), and $p=\infty$ corresponds to the linear ordering given in (i).
\end{lemma}

\begin{proof}
Definition \ref{def:co-slicing}$(1)$ means that the action of $\eta$ is determined by $\eta((n,i))=(n+1,i)$ for $i\in\{0,1\}$.
Since there are non-zero morphisms $P_{0}\rightarrow P_{1}$, we must have $(n,1) < (n, 0)$.  The determination of $\eta$ above means that $(n,i)<(m,i)$ for all integers $n<m$ and $i\in\{0,1\}$.

Consider the position of $(n+p,1)$ for some $p\in\bN$. Either, there is a $p\in\bN$ such that for each $n\in\bZ$ one has $(n+p,1)<(n+1,0)$ but $(n+p+1,1)>(n+1,0)$, or there is not. In the case that there is no such $p$, we obtain $(n,1)<(m,0)$ for all $n,m\in\bZ$ and lie in case (i) of the lemma. In case that there is such a $p\in\bN$, then one sees that this is case (ii) of the lemma.
\end{proof}

\begin{proposition} \label{prop:exceptional}
There are essential co-slicings $(\cE_{p},\{\cQ(\epsilon)\,|\,\epsilon\in\cE_{p}\})$ of $\sD^{b}(\bK Q)$ indexed by $p\in\bN\cup\{\infty\}$ defined as follows:
\begin{itemize}
\item $\cE_{p}=\bZ\times\{0,1\}$, with the linear order indexed by $p$ and automorphism $\eta:\cE_{p}\rightarrow\cE_{p}$ given by $\eta((n,i))=(n+1,i)$ for $n\in\bZ$ and $i\in\{0,1\}$.
\item $\cQ((n,0))=\add{\Sigma^{n}P_{0}}$ and $\cQ((n,1))=\add{\Sigma^{n}P_{1}}$.
\end{itemize}
Moreover, each $(\cE_{p},\{\cQ(\epsilon)\,|\,\epsilon\in\cE_{p}\})$ is a minimal essential co-slicing; see Remark \ref{rem:criterion}.
\end{proposition}

\begin{proof}
We only need to obtain Harder-Narasimhan filtrations. Consider an indecomposable object $X$ in degree $n$. Transitivity of the linear order means that we require $(n,1)<(n+1,0)$. Consider the distinguished triangle,
\begin{equation}
\Sigma^{n}P_{1}^{\oplus a}\rightarrow X\rightarrow(\Sigma^{n+1}P_{0})^{\oplus b}\rightarrow\Sigma^{n+1}P_{1}^{\oplus a},
\label{eqn:HN1}
\end{equation}
where $a,b\geq 0$ are chosen so that \eqref{eqn:HN1} is one of the Standard Triangles \ref{triangles}. This can then be re-written into the form of an HN filtration:
$$\xymatrix{0\ar[r] & \Sigma^{n}P_{1}^{\oplus a}\ar[r]\ar@{=}[d] & X\ar[d] \\
 & \Sigma^{n}P_{1}^{\oplus a}\ar@{~>}[ul] & (\Sigma^{n+1} P_{0})^{\oplus b}.\ar@{~>}[ul]}$$
Now, for $X$ not indecomposable, \cite[Lemmas 7.1 and 7.2]{Co-stability} permit the necessary re-ordering of HN quotients in order to correspond with the linear order on the co-slices.
\end{proof}

\begin{example}
Note that the behaviour of HN filtrations in the co-slicing setting is different from that for slicings. In particular, HN filtrations here are not necessarily unique. As a concrete example, taking $X=R_{x,1}$, we get the following `non-canonical' HN filtration of $R_{x,1}$ in addition to the one described above:
$$\xymatrix{0\ar[r] & P_{1}\ar[r]\ar@{=}[d] & P_{0}\oplus P_{1}\ar[r]\ar[d] & R_{x,1} = X \ar[d]. \\
 & P_{1}\ar@{~>}[ul] & P_{0}\ar@{~>}[ul]^-{0} & (\Sigma P_{0})^{\oplus 2}\ar@{~>}[ul]}$$
\end{example}

Mirroring the terminology of \cite{GKR}, we shall call the co-slicings of Proposition \ref{prop:exceptional} \emph{exceptional co-slicings}. Note the absence of a co-slicing mirror image of the `standard slicings' of \cite{GKR}; this is explained below in Proposition \ref{prop:no_regulars}.


\section{Classification of co-slicings of $\sD^{b}(\bK Q)$}\label{sec:co-slice_class}

In this section we obtain the classification of co-slicings of $\sD^{b}(\bK Q)$ analogous to that of slicings in \cite{GKR}. Note that here the proof differs significantly from that in \cite{GKR}. We start by making the following definition.

\begin{definition}
Let $(\Phi,\{\cQ(\phi)\,|\, \phi\in\Phi\})$ be a co-slicing of a triangulated category $\sT$. It is called a \emph{trivial co-slicing} if there exists $\phi\in\Phi$ such that $\cQ(\phi)=\sT$.
\end{definition}

\begin{proposition} \label{prop:no_regulars}
Suppose $(\Phi,\{\cQ(\phi)\,|\, \phi\in\Phi\})$ is a non-trivial co-slicing of $\sD^{b}(\bK Q)$ with automorphism $\lambda:\Phi\rightarrow\Phi$. An indecomposable regular object of $\sD^{b}(\bK Q)$ is not semistable with respect to $(\Phi,\{\cQ(\phi)\,|\, \phi\in\Phi\})$.
\end{proposition}

\begin{proof}
Suppose that an indecomposable regular object $R$ is semistable of phase $\phi$, i.e. $R\in\cQ(\phi)$. Assume also that, considered as a $\bK Q$-module, $R$ has regular length $d\in\bN$. From Standard Facts \ref{facts}, we have 
$\Hom{\sD^{b}(\bK Q)}{R}{\Sigma R} \cong \Ext{1}{\bK Q}(R,R) \cong \bK$.
Hence, we have $\phi\geq \lambda(\phi)$. By Definition \ref{def:co-slicing}$(1)$, we have $\phi\leq\lambda(\phi)$, so that  $\phi = \lambda(\phi)$. In particular, $\Sigma^{n} R\in\cQ(\phi)$ for all $n\in\bZ$ so that $\cQ(\phi)$ is not partial silting. Thus, by Proposition \ref{prop:silting} we get $\cQ(\phi)=\Sigma\cQ(\phi)$.

Consider an indecomposable non-regular object $N_{t}$ for $t\in\bZ$. If $N_{t}\in\cQ(\phi)$ then $\Sigma^{n} N_{t}\in\cQ(\phi)$ for all $n\in\bZ$ since $\cQ(\phi)=\Sigma\cQ(\phi)$. By taking the appropriate (de)suspension of $N_{t}$ and re-labelling so that $N_{t}$ and $R$ sit in the same degree, we obtain the following standard triangle:
$\tri{N_{t}^{\oplus d}}{N_{t+1}^{\oplus d}}{R}$. 
By closure under extensions and direct summands we find that $N_{t+1}\in\cQ(\phi)$.
Extension closure now gives that each indecomposable non-regular object belongs to the co-slice $\cQ(\phi)$, so that $(\Phi,\{\cQ(\phi)\,|\, \phi\in\Phi\})$ is a trivial co-slicing of $\sD^{b}(\bK Q)$; a contradiction. Therefore, in this case, an indecomposable non-regular object cannot be semistable.

Hence, we must have that the co-slices of $(\Phi,\{\cQ(\phi)\,|\, \phi\in\Phi\})$ consist of only regular objects. The regular components of a representation-tame hereditary algebra are extension closed. In particular, any map from an indecomposable non-regular object to a (not necessarily indecomposable) regular object must have a co-cone with a non-regular summand. Therefore, if the semistable components are required to be regular, then one cannot write down an HN filtration of an indecomposable non-regular object. Hence Definition \ref{def:co-slicing}$(3)$ is not satisfied, contradicting the assumption that $(\Phi,\{\cQ(\phi)\,|\, \phi\in\Phi\})$ is a co-slicing. Hence $R\notin \cQ(\phi)$ for any $\phi\in\Phi$.
\end{proof}

\begin{proposition} \label{prop:co-slice_form}
Suppose $(\Phi,\{\cQ(\phi)\,|\, \phi\in\Phi\})$ is a non-trivial co-slicing of $\sD^{b}(\bK Q)$ with automorphism $\lambda:\Phi\rightarrow\Phi$. A co-slice $\cQ(\phi)$ has either one, two or countably many indecomposable objects. The possibilities are as follows:
\begin{enumerate}[label=(\roman*),leftmargin=*]
\item If a co-slice $\cQ(\phi)$ has precisely one indecomposable object then $\ind{\cQ(\phi)} = \{\Sigma^{n}N_{t}\}$ for some integers $n$ and $t$. \label{item:one}
\item If a co-slice $\cQ(\phi)$ has precisely two indecomposable objects then $\ind{\cQ(\phi)} = \{\Sigma^{n}N_{t},\Sigma^{n+p}N_{t+1}\}$ for some integers $n$ and $t$ and some $p\in \bN\cup\{0\}$. \label{item:two}
\item If a co-slice $\cQ(\phi)$ has countably many indecomposable objects then $\ind{\cQ(\phi)} = \{\Sigma^{n}N_{t}\,|\, n\in\bZ\}$ for some integer $t$. \label{item:countable}
\end{enumerate}
\end{proposition}

\begin{proof}
Suppose $(\Phi,\{\cQ(\phi)\,|\, \phi\in\Phi\})$ is a non-trivial co-slicing of $\sD^{b}(\bK Q)$. By Proposition \ref{prop:silting} a non-empty co-slice $\cQ(\phi)$ is either stable under suspensions or partial silting.

Let $Q\in\cQ(\phi)$. Proposition \ref{prop:no_regulars} implies that $Q$ decomposes into a direct sum of indecomposable non-regular objects of $\sD^{b}(\bK Q)$. Without loss of generality we may assume that $P_{0}$ is a direct summand of $Q$. 

If $\cQ(\phi)=\add{P_{0}}$ then we are in case \ref{item:one} of the proposition and we are done.

So suppose that $\cQ(\phi)\neq\Sigma\cQ(\phi)$ and there is an indecomposable non-regular object $P_{0}\neq X\in\cQ(\phi)$. Examining the various possibilities for $X$, we obtain that $X$ is precisely one of $\Sigma^{i}P_{1}$ for some $i\geq 0$ or $\Sigma^{-j}I_{0}$ for some $j>0$. Note that we cannot add any further indecomposable non-regular objects since we require that $\cQ(\phi)$ is partial silting. (See, for example, \cite{AI} for a list of silting sets of the Kronecker algebra.) Hence, $\ind{\cQ(\phi)}$ belongs to item \ref{item:two} of the proposition.

Now assume $\cQ(\phi)=\Sigma\cQ(\phi)$, then we have $\{\Sigma^{n}P_{0}\,|\, n\in\bZ\}\subseteq\cQ(\phi)$. Suppose that there is a non-regular indecomposable object $X \in \cQ(\phi)$ which is not equal to any (de)suspension of $P_{0}$. Then, taking appropriate (de)suspensions we obtain that $P_{t}\in\cQ(\phi)$ for some $t>0$ or $\Sigma^{-1}I_{t}\in\cQ(\phi)$ for some $t\geq 0$. If $P_{t}\in\cQ(\phi)$, then there is a distinguished triangle 
$$\tri{P_{0}^{\oplus t-1}}{P_{1}^{\oplus t}}{P_{t}},$$
whence closure under extensions and direct summands yields $P_{1}\in\cQ(\phi)$. Now $P_{0}$ and $P_{1}$ generate $\sD^{b}(\bK Q)$, so that $\cQ(\phi)=\sD^{b}(\bK Q)$, contradicting the assumption that the co-slicing is non-trivial. Hence, we have $\cQ(\phi)=\{\Sigma^{n}P_{0}\,|\, n\in\bZ\}$, putting us in case \ref{item:countable}. Analogously if $X$ is a (de)suspension of an indecomposable preinjective module.
\end{proof}

\begin{lemma} \label{lem:two_objects}
Suppose $(\Phi,\{\cQ(\phi)\,|\, \phi\in\Phi\})$ is a non-trivial essential co-slicing of $\sD^{b}(\bK Q)$ with automorphism $\lambda:\Phi\rightarrow\Phi$. Suppose, for some $\phi\in\Phi$, we have $\ind{\cQ(\phi)} = \{\Sigma^{n}N_{t},\Sigma^{n+i}N_{t+1}\}$, where $n,t\in\bZ$ and $i\geq 0$. Then the co-slicing is completely specified by:
\begin{itemize}
\item $\Phi=\bZ$ and $\lambda(\phi)=\phi+1$; and,
\item $\cQ(\phi)=\add{\Sigma^{\phi}N_{t},\Sigma^{\phi+i}N_{t+1}}$.
\end{itemize}
\end{lemma}

\begin{proof}
Without loss of generality, we may assume that $\ind{\cQ(\phi)} = \{P_{0},\Sigma^{i}P_{1}\}$ for some $i\geq 0$. Compatibility of the linear automorphism $\lambda$ with the suspension functor forces the following linear order on $\phi$ and its images under $\lambda$ and $\lambda^{-1}$:
\begin{equation}
\cdots\lambda^{-2}(\phi)<\lambda^{-1}(\phi)<\phi<\lambda(\phi)<\lambda^{2}(\phi)<\cdots.
\label{eqn:linear_order}
\end{equation}

Now suppose that $X\neq \Sigma^{n}P_{0}$ and $X\neq \Sigma^{n}P_{1}$ for any $n\in\bZ$. Suppose further that $X\in\cQ(\psi)$ for some $\psi\in\Phi$. Again, by taking appropriate (de)suspensions and relabelling $\psi$ accordingly, we may assume that either $X=P_{t}$ for some $t>1$ or $X=I_{t}$ for some $t\geq 0$. If $X=P_{t}$ there are non-zero maps $P_{0}\rightarrow P_{t}$ and $P_{t}\rightarrow \Sigma P_{0}$. The first map implies that $\phi\geq\psi$ and the second one implies that $\psi\geq\lambda(\phi)$, by Definition \ref{def:co-slicing}$(1)$. Putting these together, we get $\phi\geq\lambda(\phi)$; a contradiction. A similar argument holds if $X=I_{t}$ for some $t\geq 0$. Hence, $X$ cannot be semistable.

Thus, the only co-slices are the (de)suspensions of $\cQ(\phi)$. Hence, the linearly ordered set in \eqref{eqn:linear_order} is all of $\Phi$.
\end{proof}

The following corollary is straightforward:

\begin{corollary} \label{cor:one_object}
Suppose $(\Phi,\{\cQ(\phi)\,|\, \phi\in\Phi\})$ is a non-trivial essential co-slicing of $\sD^{b}(\bK Q)$ with automorphism $\lambda:\Phi\rightarrow\Phi$. 
\begin{enumerate}[label=(\roman*),leftmargin=*]
\item If $N_{t}\in\cQ(\phi)$ for some $\phi\in\Phi$, then the only other possible semistable indecomposable objects of any phase are precisely the (de)suspensions of $N_{t-1}$ and $N_{t+1}$. Moreover, at most one of $N_{t-1}$ and $N_{t+1}$ can be semistable.
\item Proposition \ref{prop:exceptional} gives a complete list of the essential co-slicings of $\sD^{b}(\bK Q)$ in which each co-slice contains only one indecomposable object.
\end{enumerate}
\end{corollary}

\begin{theorem} \label{thm:comparison}
Suppose $(\Phi,\{\cR(\phi)\,|\, \phi\in\Phi\})$ is an essential co-slicing of $\sD^{b}(\bK Q)$ with automorphism $\lambda:\Phi\rightarrow\Phi$. Then $(\Phi,\{\cR(\phi)\,|\, \phi\in\Phi\})$ is coarser than one of the exceptional co-slicings $\cE_{p}$ for $p\in\bN\cup\{\infty\}$.
\end{theorem}

\begin{proof}
Suppose $(\Phi,\{\cR(\phi)\,|\, \phi\in\Phi\})$ is an essential co-slicing of $\sD^{b}(\bK Q)$ with automorphism $\lambda:\Phi\rightarrow\Phi$. In Proposition \ref{prop:co-slice_form} we established all the possible forms for each co-slice. Suppose each co-slice contains only one indecomposable object. Then by Corollary \ref{cor:one_object} it is one of the exceptional co-slicings $\cE_{p}$ constructed in Proposition \ref{prop:exceptional}.

Now suppose there is a co-slice with two distinct indecomposable objects. Then by Lemma \ref{lem:two_objects}, the co-slicing is given by
\begin{itemize}
\item $\Phi=\bZ$ and $\lambda(\phi)=\phi+1$; and,
\item $\cR(\phi)=\add{\Sigma^{\phi}N_{t},\Sigma^{\phi+p-1}N_{t+1}}$ for some $p\in\bN$.
\end{itemize}
Define a map $r:\cE_{p}=\bZ\times\{0,1\}\rightarrow \bZ$ by 
$$(n,0) \mapsto  n \quad {\rm and} \quad (n,1)  \mapsto  n-p+1.$$
This map is clearly surjective and it is easy to show it satisfies Definition \ref{def:finer}(ii). Definition \ref{def:finer}(iii) can be seen by noting for $n\in\bZ$, we have $r^{-1}(n)=\{(n,0),(n+p-1,1)\}$. It follows that 
$$(\bigcup_{\epsilon\in r^{-1}(n)} \cQ(\epsilon))^{+} = \add{\Sigma^{n}N_{t},\Sigma^{n+p-1}N_{t+1}} = \cR(n).$$
Hence, the co-slicing $(\Phi,\{\cR(\phi)\,|\, \phi\in\Phi\})$ is coarser than the exceptional co-slicing $(\cE_{p},\{\cQ(\phi)\,|\, \phi\in\Phi\})$.

The final case to check is when a co-slice has infinitely many indecomposable objects, i.e takes the form \ref{item:countable} in Proposition \ref{prop:co-slice_form}.
Suppose that one co-slice has the form $\add{\Sigma^{n}N_{t}\,|\, n\in\bZ}$ for some $t\in\bZ$. Then either
\begin{enumerate}[label=(\alph*),leftmargin=*]
\item there is only one other co-slice, without loss of generality we may assume it to be $\add{\Sigma^{n}N_{t+1}\,|\, n\in\bZ}$; or
\item there are countably many other co-slices, without loss of generality, we may assume them to be $\add{\Sigma^{n}N_{t+1}}$ for $n\in\bZ$.
\end{enumerate}
We shall show that in both cases they are coarser than (the same) exceptional co-slicing of $\sD^{b}(\bK Q)$.
In case (a), the co-slicing is specified by
\begin{itemize}
\item $\Phi=\{0,1\}$ and $\lambda(\phi)=\phi$; and,
\item $\cR(0)=\add{\Sigma^{n}N_{t}\,|\, n\in\bZ}$ and $\cR(1)=\add{\Sigma^{n}N_{t+1}\,|\, n\in\bZ}$.
\end{itemize}
One can then easily check that $r:\cE_{\infty} = \bZ\times\{0,1\}\rightarrow \{0,1\}$ defined by $(n,i)\mapsto i$ is a surjective map satisfying conditions (i), (ii) and (iii) of Definition \ref{def:finer}.

In case (b), the co-slicing is specified by
\begin{itemize}
\item $\Phi=\bZ\cup\{\infty\}$ with $n<\infty$ for all $n\in\bZ$. The linear automorphism acts as $\lambda(n)=n+1$ for $n\in\bZ$ and $\lambda(\infty)=\infty$; and,
\item $\cR(n)=\add{\Sigma^{n}N_{t+1}}$ for $n\in\bZ$ and $\cR(\infty)=\add{\Sigma^{n}N_{t}\,|\, n\in\bZ}$.
\end{itemize}
Again, one can easily check that $r:\cE_{\infty} = \bZ\times\{0,1\}\rightarrow \bZ\cup\{\infty\}$ defined by 
$$(n,0)  \mapsto  \infty \quad {\rm and} \quad (n,1)  \mapsto  n,$$
is a surjective map satisfying the conditions of Definition \ref{def:finer}.

Hence, each co-slicing of $\sD^{b}(\bK Q)$ is coarser than one of the exceptional co-slicings of Proposition \ref{prop:exceptional}.
\end{proof}


\section{Classification of co-t-structures in $\sD^{b}(\bK Q)$}
\label{sec:classification}

In this section we use Theorem \ref{thm:comparison} together with the exceptional co-slicings of Section \ref{sec:exceptional} to obtain a classification of the co-t-structures in $\sD^{b}(\bK Q)$. Note that we obtain all co-t-structures in $\sD^{b}(\bK Q)$, not just the bounded co-t-structures.

In Theorem~\ref{thm:all} we shall show that the co-t-structures in $\sD^{b}(\bK Q)$ are induced by the exceptional co-slicings of Proposition \ref{prop:exceptional}. Note that Theorem \ref{thm:comparison} means that all essential co-slicings of $\sD^{b}(\bK Q)$ are comparable with those of Proposition \ref{prop:exceptional}. As in Section \ref{sec:exceptional}, we explain the classification only for the exceptional pair $\{P_{0},P_{1}\}$. The other co-t-structures can be obtained analogously. 

The induced co-t-structures occur in four families outlined below, one bounded family and three unbounded families.

\subsection{The family of bounded co-t-structures}
\label{sec:bounded}

This family is induced by the minimal co-slicings $\cE_{p}$ for $p\in\bN$ of Lemma \ref{lem:linear_order}(ii) and Proposition \ref{prop:exceptional}. The only way to partition $\cE_{p}$ as $\cE_{p}^{-}\cup\cE_{p}^{+}$ with $\epsilon_{-}<\epsilon_{+}$ for all $\epsilon_{-}\in\cE_{p}^{-}$ and $\epsilon_{+}\in\cE_{p}^{+}$ is to choose a $\delta\in\cE_{p}$ and define $\cE_{p}^{-}=\{\epsilon\,|\,\epsilon\leq\delta \}$ and $\cE_{p}^{+}=\{\epsilon\,|\,\epsilon>\delta\}$. We get the co-t-structure  $(\sA_{\delta}(\cE_{p}),\sB_{\delta}(\cE_{p}))$ defined by
\[
\sA_{\delta}(\cE_{p})  := (\bigcup_{\epsilon\leq\delta}\cQ(\epsilon))^{+} \quad {\rm and} \quad \sB_{\delta}(\cE_{p})  :=  (\bigcup_{\epsilon>\delta}\cQ(\epsilon))^{+}.
\]

We indicate the case $\delta=(0,0)$ and $p=3$ on a sketch of the Auslander-Reiten quiver of $\sD^{b}(\bK Q)$. The full subcategory $\sA_{\delta}(\cE_{p})$ is indicated by horizontal hatching
\begin{tikzpicture}
\draw [pattern=A hatch](0,0) rectangle (0.5,0.3);
\end{tikzpicture}
and the full subcategory $\sB_{\delta}(\cE_{p})$ is indicated by vertical hatching 
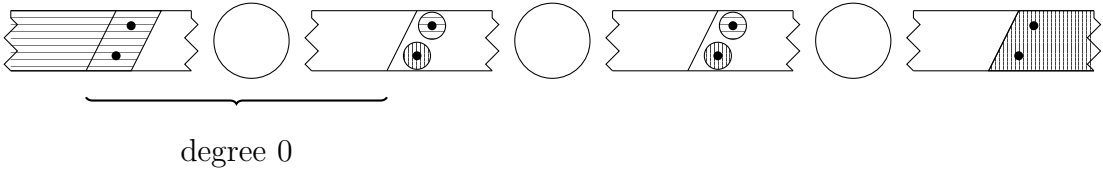
\begin{tikzpicture}
\draw [pattern=B hatch](0,0) rectangle (0.5,0.3);
\end{tikzpicture}.

\

\begin{center}
\begin{tikzpicture}
\draw (0,0) -- (2.4,0);
\draw (0,0.8) -- (2.4,0.8);
\draw [decorate,decoration=zigzag] (0,0) -- (0, 0.8);
\draw [decorate,decoration=zigzag] (2.4,0) -- (2.4, 0.8);
\draw (1,0) -- (1.4, 0.8);

\fill (1.6,0.6) circle (0.6mm);
\fill (1.4,0.2) circle (0.6mm);

\draw[pattern=A hatch] (0,0) -- (1.6,0) -- (2.0,0.8) -- (0,0.8);

\draw (3.2,0.4) circle (5mm);

\draw[thick,decorate,decoration={brace,mirror,raise=10pt}] (1,0) -- (5,0) node[pos=0.5,anchor=north,yshift=-0.75cm] {degree 0};

\draw (4.0,0) -- (6.4,0);
\draw (4,0.8) -- (6.4,0.8);
\draw [decorate,decoration=zigzag] (4.0,0) -- (4.0, 0.8);
\draw [decorate,decoration=zigzag] (6.4,0) -- (6.4, 0.8);
\draw (5,0) -- (5.4, 0.8);

\fill (5.6,0.6) circle (0.6mm);
\fill (5.4,0.2) circle (0.6mm);

\draw[pattern=A hatch] (5.6,0.6) circle (1.8mm);
\draw[pattern=B hatch] (5.4,0.2) circle (1.8mm);

\draw (7.2,0.4) circle (5mm);

\draw (8.0,0) -- (10.4,0);
\draw (8.0,0.8) -- (10.4,0.8);
\draw [decorate,decoration=zigzag] (8.0,0) -- (8.0, 0.8);
\draw [decorate,decoration=zigzag] (10.4,0) -- (10.4, 0.8);
\draw (9,0) -- (9.4, 0.8);

\fill (9.6,0.6) circle (0.6mm);
\fill (9.4,0.2) circle (0.6mm);

\draw[pattern=A hatch] (9.6,0.6) circle (1.8mm);
\draw[pattern=B hatch] (9.4,0.2) circle (1.8mm);

\draw (11.2,0.4) circle (5mm);

\draw (12.0,0) -- (14.4,0);
\draw (12.0,0.8) -- (14.4,0.8);
\draw [decorate,decoration=zigzag] (12.0,0) -- (12.0, 0.8);
\draw [decorate,decoration=zigzag] (14.4,0) -- (14.4, 0.8);
\draw (13,0) -- (13.4, 0.8);

\draw[pattern=B hatch] (14.4,0) -- (13,0) -- (13.4,0.8) -- (14.4,0.8);

\fill (13.6,0.6) circle (0.6mm);
\fill (13.4,0.2) circle (0.6mm);
\end{tikzpicture}
\end{center}

\subsection{The families of unbounded co-t-structures}
\label{sec:unbounded}

The remaining families are induced by the exceptional co-slicing with the linear ordering $\cE_{\infty}$ of Lemma \ref{lem:linear_order}. Recall, the linear ordering is given by $(n,i)<(m,i)$ for $i\in\{0,1\}$ and $n<m$, and $(n,1)<(m,0)$ for any $n,m\in\bZ$.

The co-t-structures that are bounded on one side are induced by $\cE_{\infty}$ by choosing a $\delta \in \cE_{\infty}$ and defining  $(\sA_{\delta}(\cE_{\infty}),\sB_{\delta}(\cE_{\infty}))$ by
\[
\sA_{\delta}(\cE_{\infty})  := (\bigcup_{\epsilon\leq\delta}\cQ(\epsilon))^{+} \quad {\rm and} \quad \sB_{\delta}(\cE_{\infty})  :=  (\bigcup_{\epsilon>\delta}\cQ(\epsilon))^{+}.
\]

\subsubsection{The subfamily of bounded below co-t-structures}
\label{sec:bdd_below}

 The first subfamily is the family of bounded below co-t-structures, which is obtained by taking $\delta = (n,0)$ for some $n\in \bZ$. 
Below, we sketch $(\sA_{(0,0)}(\cE_{\infty}),\sB_{(0,0)}(\cE_{\infty}))$:

\

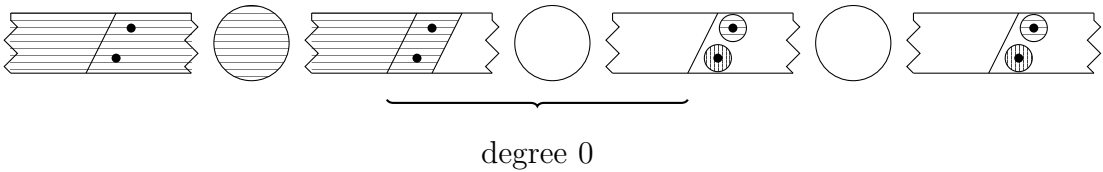
\begin{center}
\begin{tikzpicture}
\draw (0,0) -- (2.4,0);
\draw (0,0.8) -- (2.4,0.8);
\draw [decorate,decoration=zigzag] (0,0) -- (0, 0.8);
\draw [decorate,decoration=zigzag] (2.4,0) -- (2.4, 0.8);
\draw (1,0) -- (1.4, 0.8);

\path [pattern=A hatch] (0,0) -- (2.4,0) -- (2.4, 0.8) -- (0,0.8);

\fill (1.6,0.6) circle (0.6mm);
\fill (1.4,0.2) circle (0.6mm);

\draw [pattern=A hatch] (3.2,0.4) circle (5mm);

\draw (4.0,0) -- (6.4,0);
\draw (4,0.8) -- (6.4,0.8);
\draw [decorate,decoration=zigzag] (4.0,0) -- (4.0, 0.8);
\draw [decorate,decoration=zigzag] (6.4,0) -- (6.4, 0.8);
\draw (5,0) -- (5.4, 0.8);

\fill (5.6,0.6) circle (0.6mm);
\fill (5.4,0.2) circle (0.6mm);

\draw[thick,decorate,decoration={brace,mirror,raise=10pt}] (5,0) -- (9,0) node[pos=0.5,anchor=north,yshift=-0.75cm] {degree 0};

\path[pattern=A hatch] (4.0,0) -- (5.6,0) -- (6,0.8) -- (4.0,0.8);
\draw (5.6,0) -- (6,0.8);

\draw (7.2,0.4) circle (5mm);

\draw (8.0,0) -- (10.4,0);
\draw (8.0,0.8) -- (10.4,0.8);
\draw [decorate,decoration=zigzag] (8.0,0) -- (8.0, 0.8);
\draw [decorate,decoration=zigzag] (10.4,0) -- (10.4, 0.8);
\draw (9,0) -- (9.4, 0.8);

\fill (9.6,0.6) circle (0.6mm);
\fill (9.4,0.2) circle (0.6mm);

\draw[pattern=A hatch] (9.6,0.6) circle (1.8mm);
\draw[pattern=B hatch] (9.4,0.2) circle (1.8mm);

\draw (11.2,0.4) circle (5mm);

\draw (12.0,0) -- (14.4,0);
\draw (12.0,0.8) -- (14.4,0.8);
\draw [decorate,decoration=zigzag] (12.0,0) -- (12.0, 0.8);
\draw [decorate,decoration=zigzag] (14.4,0) -- (14.4, 0.8);
\draw (13,0) -- (13.4, 0.8);

\fill (13.6,0.6) circle (0.6mm);
\fill (13.4,0.2) circle (0.6mm);

\draw[pattern=A hatch] (13.6,0.6) circle (1.8mm);
\draw[pattern=B hatch] (13.4,0.2) circle (1.8mm);

\end{tikzpicture}
\end{center}

\subsubsection{The subfamily of bounded above co-t-structures}
\label{sec:bdd_above}

The second subfamily consists of the bounded above co-t-structure, which are obtained by taking $\delta = (n,1)$ for some $n \in \bZ$.
Below, we sketch $(\sA_{(0,1)}(\cE_{\infty}),\sB_{(0,1)}(\cE_{\infty}))$:

\

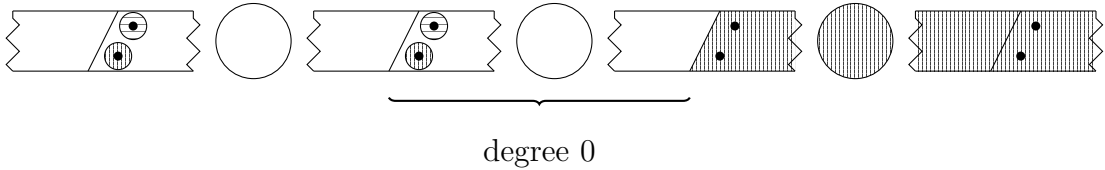
\begin{center}
\begin{tikzpicture}
\draw (0,0) -- (2.4,0);
\draw (0,0.8) -- (2.4,0.8);
\draw [decorate,decoration=zigzag] (0,0) -- (0, 0.8);
\draw [decorate,decoration=zigzag] (2.4,0) -- (2.4, 0.8);
\draw (1,0) -- (1.4, 0.8);

\fill (1.6,0.6) circle (0.6mm);
\fill (1.4,0.2) circle (0.6mm);
\draw[pattern=A hatch] (1.6,0.6) circle (1.8mm);
\draw[pattern=B hatch] (1.4,0.2) circle (1.8mm);

\draw (3.2,0.4) circle (5mm);

\draw (4.0,0) -- (6.4,0);
\draw (4,0.8) -- (6.4,0.8);
\draw [decorate,decoration=zigzag] (4.0,0) -- (4.0, 0.8);
\draw [decorate,decoration=zigzag] (6.4,0) -- (6.4, 0.8);
\draw (5,0) -- (5.4, 0.8);

\fill (5.6,0.6) circle (0.6mm);
\fill (5.4,0.2) circle (0.6mm);
\draw[pattern=A hatch] (5.6,0.6) circle (1.8mm);
\draw[pattern=B hatch] (5.4,0.2) circle (1.8mm);

\draw[thick,decorate,decoration={brace,mirror,raise=10pt}] (5,0) -- (9,0) node[pos=0.5,anchor=north,yshift=-0.75cm] {degree 0};

\draw (7.2,0.4) circle (5mm);

\draw (8.0,0) -- (10.4,0);
\draw (8.0,0.8) -- (10.4,0.8);
\draw [decorate,decoration=zigzag] (8.0,0) -- (8.0, 0.8);
\draw [decorate,decoration=zigzag] (10.4,0) -- (10.4, 0.8);
\draw (9,0) -- (9.4, 0.8);

\fill (9.6,0.6) circle (0.6mm);
\fill (9.4,0.2) circle (0.6mm);

\path [pattern=B hatch] (10.4,0.8) -- (9.4,0.8) -- (9,0) -- (10.4,0);
\path [pattern=B hatch] (12.0,0) -- (14.4,0) -- (14.4,0.8) -- (12.0,0.8);

\draw [pattern=B hatch] (11.2,0.4) circle (5mm);

\draw (12.0,0) -- (14.4,0);
\draw (12.0,0.8) -- (14.4,0.8);
\draw [decorate,decoration=zigzag] (12.0,0) -- (12.0, 0.8);
\draw [decorate,decoration=zigzag] (14.4,0) -- (14.4, 0.8);
\draw (13,0) -- (13.4, 0.8);

\fill (13.6,0.6) circle (0.6mm);
\fill (13.4,0.2) circle (0.6mm);
\end{tikzpicture}
\end{center}

\subsubsection{The subfamily of stable co-t-structures}
\label{sec:stable}

The stable co-t-structure (= stable t-structure), which we denote $(\sA(\cE_{\infty}),\sB(\cE_{\infty}))$ is defined by
$$\sA(\cE_{\infty}) := (\bigcup_{\epsilon\in \cE_{\infty}^{-}} \cQ(\epsilon) )^{+} \quad {\rm and} \quad 
\sB(\cE_{\infty})  :=  (\bigcup_{\epsilon\in \cE_{\infty}^{+}}\cQ(\epsilon))^{+}.$$
Where $\cE_{\infty}^{-}=\{(n,1)\,|\, n\in\bZ\}$ and $\cE_{\infty}^{+}=\{(n,0)\,|\, n\in\bZ\}$. We sketch the stable co-t-structure below.

\

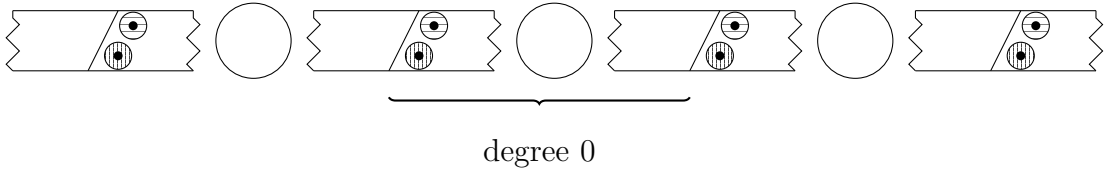
\begin{center}
\begin{tikzpicture}
\draw (0,0) -- (2.4,0);
\draw (0,0.8) -- (2.4,0.8);
\draw [decorate,decoration=zigzag] (0,0) -- (0, 0.8);
\draw [decorate,decoration=zigzag] (2.4,0) -- (2.4, 0.8);
\draw (1,0) -- (1.4, 0.8);

\fill (1.6,0.6) circle (0.6mm);
\fill (1.4,0.2) circle (0.6mm);
\draw[pattern=A hatch] (1.6,0.6) circle (1.8mm);
\draw[pattern=B hatch] (1.4,0.2) circle (1.8mm);

\draw (3.2,0.4) circle (5mm);

\draw (4.0,0) -- (6.4,0);
\draw (4,0.8) -- (6.4,0.8);
\draw [decorate,decoration=zigzag] (4.0,0) -- (4.0, 0.8);
\draw [decorate,decoration=zigzag] (6.4,0) -- (6.4, 0.8);
\draw (5,0) -- (5.4, 0.8);

\fill (5.6,0.6) circle (0.6mm);
\fill (5.4,0.2) circle (0.6mm);
\draw[pattern=A hatch] (5.6,0.6) circle (1.8mm);
\draw[pattern=B hatch] (5.4,0.2) circle (1.8mm);

\draw[thick,decorate,decoration={brace,mirror,raise=10pt}] (5,0) -- (9,0) node[pos=0.5,anchor=north,yshift=-0.75cm] {degree 0};

\draw (7.2,0.4) circle (5mm);

\draw (8.0,0) -- (10.4,0);
\draw (8.0,0.8) -- (10.4,0.8);
\draw [decorate,decoration=zigzag] (8.0,0) -- (8.0, 0.8);
\draw [decorate,decoration=zigzag] (10.4,0) -- (10.4, 0.8);
\draw (9,0) -- (9.4, 0.8);

\fill (9.6,0.6) circle (0.6mm);
\fill (9.4,0.2) circle (0.6mm);

\draw[pattern=A hatch] (9.6,0.6) circle (1.8mm);
\draw[pattern=B hatch] (9.4,0.2) circle (1.8mm);

\draw (11.2,0.4) circle (5mm);

\draw (12.0,0) -- (14.4,0);
\draw (12.0,0.8) -- (14.4,0.8);
\draw [decorate,decoration=zigzag] (12.0,0) -- (12.0, 0.8);
\draw [decorate,decoration=zigzag] (14.4,0) -- (14.4, 0.8);
\draw (13,0) -- (13.4, 0.8);

\fill (13.6,0.6) circle (0.6mm);
\fill (13.4,0.2) circle (0.6mm);

\draw[pattern=A hatch] (13.6,0.6) circle (1.8mm);
\draw[pattern=B hatch] (13.4,0.2) circle (1.8mm);
\end{tikzpicture}
\end{center}

\subsection{The classification}

We now prove the main result of this section.

\begin{theorem} \label{thm:all}
Every co-t-structure in $\sD^b(\bK Q)$ is induced by an exceptional co-slicing of $\sD^b(\bK Q)$. In particular, up to (de)suspension, all co-t-structures in $\sD^b(\bK Q)$ arise from exceptional pairs of $\mod{\bK Q}$ (or, equivalently $\mathsf{coh}(\bP^1)$) in the manner described in Sections \ref{sec:bounded} and \ref{sec:unbounded}.
\end{theorem}

\begin{proof}
First note that by \cite[Lemma 3.1(iv)]{HJY}, the co-heart $\sC$ of a co-t-structure $(\sA,\sB)$ is zero if and only if the co-t-structure is stable. By Example \ref{ex:stable}, all stable co-t-structures are, in fact, examples of co-slicings, and in particular are coarser than the co-slicings of Theorem \ref{thm:comparison}, and hence the co-t-structures they induce occur in the list in Section \ref{sec:stable} above.

Now suppose $\sC \neq 0$ and consider the possibilites for the size of the co-heart $\sC$. We must have $|\ind\sC|=1$ or $2$, since $\sC$ is partial silting, see \cite[Theorem 2.26]{AI}. If $|\ind\sC|=2$ then $\sC$ is silting, indeed its indecomposable objects are even a full exceptional collection which is well-known to generate  $\sD^b(\bK Q)$ (see \cite{GKR}, for instance). Hence, we have a bounded co-t-structure. Example \ref{ex:bounded} now guarantees that we have all bounded co-t-structures in the list in Section \ref{sec:bounded}.

Now suppose that $\ind \sC = \{c\}$. Since $\sC$ is partial silting, we may assume that $c= N_t$ for some $t\in \bZ$. The argument of Proposition~\ref{prop:co-slice_form} shows that the silting subcategories containing $N_t$ are $\add{N_t, \Sigma^i N_{t+1}}$ for some $i\geq 0$ and $\add{\Sigma^j N_{t-1}, N_t}$ for some $j \leq 0$. 
Thus, $\sC=\add{N_t}$ is an almost silting subcategory; consider the completion $\add{\Sigma^j N_{t-1}, N_t}$ for some $j\leq 0$. This satisfies the hypotheses of Propostion~\ref{prop:silting-bounded}(i), and hence the co-t-structure $(\sA,\sB)$ is bounded above or below.

Suppose $(\sA,\sB)$ is bounded above: In this case $\Sigma^i(\Sigma^j N_{t-1}) \in \sB$ for some $i\in \bZ$, so by Proposition~\ref{prop:silting-bounded}(i) $\Sigma^i(\Sigma^j N_{t-1}) \in \sB$ for all $i \in \bZ$. By Corollary~\ref{cor:silting-bounded} such a co-t-structure induces a co-slicing, which is then coarser than an exceptional co-slicing of Theorem~\ref{thm:comparison}. By Corollary~\ref{cor:silting-bounded-induced}, this co-slicing induces the co-t-structure $(\sA,\sB)$, and hence this co-t-structure appears in the list in Subsection~\ref{sec:bdd_above}.

Now suppose that $(\sA,\sB)$ is bounded below: In this case we consider the location of the object $N_{t+1}$ instead. We have $\Sigma^k N_{t+1} \in \sA$ for some $k \in \bZ$ because $(\sA,\sB)$ is bounded below. Now the object $N_{t+1}$ completes $\add{N_t}$ to a silting subcategory which satisfies the assumptions of Proposition~\ref{prop:silting-bounded}(ii). In particular, it follows from Proposition~\ref{prop:silting-bounded}(ii) that $\Sigma^k N_{t+1} \in \sA$ for all $k \in \bZ$. Using the dual of Corollary~\ref{cor:silting-bounded} and arguing as in the paragraph above shows that such a co-t-structure appears in the list in Subsection~\ref{sec:bdd_below}.
\end{proof}


\section{The co-stability manifold}
\label{sec:manifold}

Before computing the co-stability manifold of $\sD^b(\bK Q)$, we give a brief recap of pertinent facts from \cite{Co-stability}. It is defined in a manner analogous to Bridgeland's stability manifold \cite{Bridgeland}. 

\subsection{Co-stability conditions} \label{sec:recap-costability}
For this short section $\sT$ will denote an arbitrary Krull-Schmidt triangulated category. We shall make statements only for Krull-Schmidt categories, for the definitions in greater generality, see \cite{Co-stability}. All definitions and statements below are taken from \cite{Co-stability}.

In this section, we use the term \emph{co-slicing} in the sense of \cite[Definition 3.1]{Co-stability}, i.e. a co-slicing $\cQ$ will denote the generalised co-slicing $(\Phi,\{\cQ(\phi) \mid \phi \in \Phi\})$ where $\Phi = \bR$ and the automorphism of ordered sets $\lambda \colon \bR \to \bR$ is given by $\phi \mapsto \phi + 1$; see Example~\ref{ex:co-slice}. Note that, by \cite[Remark 3.8]{Co-stability}, such co-slicings are not essential in the sense of Definition~\ref{def:essential}.

A \emph{co-stability condition} on $\sT$ is a pair $(Z,\cQ)$, where $Z\colon K_0(\sT) \to \bC$ is a group homomorphism and $\cQ$ is a co-slicing of $\sT$ such that
\[
q\in \cQ(\phi), q \ncong 0 \implies Z(q) = m(q)\exp(i\pi \phi),
\]
with $m(q) >0$. The number $\phi$ is called the \emph{phase} of $q$.

There is a useful characterisation of co-stability conditions in terms of bounded co-t-structures and co-stability functions. 
A \emph{co-stability function} on an additive category $\sS$ is a group homomorphism
\[
Z\colon K_0^{\rm split}(\sS) \to \bC
\]
such that for each object $s\ncong 0$ we have $Z(s) \in H = \{r\exp(i \pi \phi) \mid r>0, 0<\phi \leq 1\}$. An object $0 \ncong s\in \sS$ is called \emph{$Z$-semistable} if and only if its indecomposable summands have the same phase.
We then have:

\begin{proposition}[\cite{Co-stability} Proposition 6.3] \label{prop:HN}
Giving a co-stability condition on $\sT$ is equivalent to giving a bounded co-t-structure in $\sT$ and a co-stability function $Z$ on its co-heart $\sC$ which satisfies the \emph{split Harder-Narasimhan property}: if $0 \ncong c_1,c_2 \in \sC$ are $Z$-semistable with $\phi(c_1) < \phi(c_2)$ then $\Hom{\sC}{c_1}{c_2}=0$. 
\end{proposition}

A co-slicing $\cQ$ of a triangulated category $\sT$ is said to satisfy condition (S) if for any two indecomposable objects $q_1, q_2 \in \cQ(\phi)$ with $q_1 \ncong q_2$ then $\Hom{\sT}{q_1}{q_2}=0$. Denote by $\Coslice{\sT}$ the set of co-slicings of $\sT$ satisfying condition (S). Note that these co-slicings are not essential, see \cite[Remark 3.8]{Co-stability}.

Recall from \cite[Proposition 4.3]{Co-stability} that for two co-slicings $\cQ,\cR \in \Coslice{\sT}$, the function
\[
d(\cQ,\cR) = \inf \{\epsilon >0 \mid \cQ(\phi) \subseteq \cR([\phi-\epsilon,\phi+\epsilon]) \textrm{ for each } \phi \in \bR\}
\]
defines a metric on $\Coslice{\sT}$. 

Now consider the product space $K_0(\sT)^* \times \Coslice{\sT}$.
The \emph{co-stability manifold} of $\sT$ is the topological subspace
\[
\Costab{\sT} \subseteq K_0(\sT)^* \times \Coslice{\sT},
\]
consisting of co-stability conditions $(Z,\cQ)$.

Recall also the following principal result from \cite{Co-stability}:

\begin{theorem}[\cite{Co-stability} Theorem 8.2]
The topological space $\Costab{\sT}$ is a topological manifold which, if non-empty, has dimension $2\cdot \textrm{rank}\, K_0(\sT)$.
\end{theorem}

We are now ready to compute the co-stability manifold of $\sD^b(\bK Q)$.

\subsection{The co-stability manifold of $\sD^b(\bK Q)$} 

Now let $\sT=\sD^b(\bK Q)$ and consider the following parametrisation of co-stability conditions on $\sD^{b}(\bK Q)$.

\begin{lemma} \label{lem:parametrisation}
There is a bijection 
$$f:\bZ\times\{(x,y)\in\bR^{2}\,|\,x<y\}\times\bR_{+}\times\bR_{+}\rightarrow\Costab{\sD^{b}(\bK Q)}.$$ The quintuple $(n,\phi_{1},\phi_{0},m_{1},m_{0})$ is sent to the co-stability condition, $(Z,\cQ)$, whose co-slices are given by
$$\cQ(\phi_{0})=\add{N_{n}} \quad and \quad \cQ(\phi_{1})=\add{N_{n+1}},$$
and whose values of $Z$ are determined by
$$Z(N_{n})=m_{0}\exp(i\pi\phi_{0}) \quad and \quad Z(N_{n+1})=m_{1}\exp(i\pi\phi_{1}).$$
\end{lemma}

\begin{proof}
It is clear that the assignment in the statement of the lemma determines a co-stability condition on $\sD^{b}(\bK Q)$. 

We shall show that a co-stability condition on $\sD^{b}(\bK Q)$ determines such a quintuple by using Proposition~\ref{prop:HN} and the fact that any bounded co-t-structure is determined by its co-heart. 

The co-hearts of the bounded co-t-structures can be determined easily from Section \ref{sec:bounded}. In particular, each bounded co-t-structure is uniquely determined by a triple $(m,n,p)\in\bZ\times\bZ\times\bN\cup\{0\}$, where the triple $(m,n,p)$ determines the co-t-structure with co-heart $\sC(m,n,p)$ whose indecomposable objects are $\Sigma^{m}N_{n}$ and $\Sigma^{m+p}N_{n+1}$. 

A co-stability function $Z$ on $\sC(m,n,p)$ is determined by the values
$$Z(\Sigma^{m}N_{n})=m_{0}\exp(i\pi\phi_{0}') \quad {\rm and} \quad Z(\Sigma^{m+p}N_{n+1})=m_{1}\exp(i\pi\phi_{1}'),$$
with $\phi_{0}',\phi_{1}'\in (0,1]$ and $m_{0},m_{1}\in\bR_{+}$. Note that in the case that $p=0$, we must have $\phi_{1}'<\phi_{0}'$. Now, setting $\phi_{0}=\phi_{0}'-m$ and $\phi_{1}=\phi_{1}'-m-p$, we obtain a pair $(\phi_{1},\phi_{0})\in\bR^{2}$ satisfying $\phi_{1}<\phi_{0}$. Together, these values give us the desired quintuple, thus, giving the desired bijection.
\end{proof}

\begin{definitions} \label{def:component}
Let $\cQ$ be a co-slicing of $\sD^{b}(\bK Q)$. Write $\ss{\cQ}$ for the semistable indecomposable objects of the co-slicing $\cQ$.

Let $\sC(n)$ be the subset of $\Costab{\sD^{b}(\bK Q)}$ consisting of co-stability conditions $(Z,\cQ)$ such that $\ss{\cQ}=\{\Sigma^{i}N_{n},\Sigma^{j}N_{n+1}\,|\, i,j\in\bZ\}$.
\end{definitions}

We now have the following lemma.

\begin{lemma} \label{lem:continuous}
Let $n\in\bZ$ and consider $f_{n}:\{(x,y)\in\bR^{2}\,|\,x<y\}\times\bR_{+}\times\bR_{+}\rightarrow\sC(n)$, given as in Lemma \ref{lem:parametrisation}. Then $f_{n}$ is continuous.
\end{lemma}

\begin{proof}
Considering the composition of $f_{n}$ with the canonical projections from $\sC(n)\rightarrow\Coslice{\sD^{b}(\bK Q)}$ and $\sC(n)\rightarrow K_{0}(\sD^{b}(\bK Q))^{*}$, one can easily see that $f_{n}$ is continuous.
\end{proof}

\begin{lemma} \label{lem:component}
Let $(Z,\cQ)$ and $(W,\cR)$ be two co-stability conditions on $\sD^{b}(\bK Q)$. $(Z,\cQ)$ and $(W,\cR)$ are in the same path component of $\Costab{\sD^{b}(\bK Q)}$ if and only if their semistable indecomposable objects coincide (possibly with different phases).
\end{lemma}

\begin{proof}
By Lemma \ref{lem:parametrisation} the co-stability conditions $(Z,\cQ)$ and $(W,\cR)$  are given by quintuples $(n,\phi_{1},\phi_{0},m_{1},m_{0})$ and $(n',\phi_{1}',\phi_{0}',m_{1}',m_{0}')$.

If the semistable indecomposable objects of $(Z,\cQ)$ and $(W,\cR)$ coincide, then $n=n'$. It is clear that the $\phi_{i}$ and $m_{i}$ can be varied continuously to obtain $\phi_{i}'$ and $m_{i}'$ for $i=0,1$. Using Lemma \ref{lem:continuous}, one obtains that $(Z,\cQ)$ and $(W,\cR)$ lie in the same path component.

Conversely, suppose $(Z,\cQ)$ and $(W,\cR)$ satisfy $\ss{\cQ}\neq\ss{\cR}$ but lie in the same path component of $\Costab{\sD^{b}(\bK Q)}$. Then there is a path $\pi:[0,1]\rightarrow\Costab{\sD^{b}(\bK Q)}$ from $(Z,\cQ)$ to $(W,\cR)$. Write $\pi':[0,1]\rightarrow\Coslice{\sD^{b}(\bK Q)}$ for the composition of $\pi$ with the canonical projection $\Costab{\sD^{b}(\bK Q)}\rightarrow\Coslice{\sD^{b}(\bK Q)}$, and set $\cQ_{t}:=\pi'(t)$ for $t\in[0,1]$.

Since $\ss{\cQ}\neq\ss{\cR}$, there exists $t_{0}\in (0,1]$ such that $\ss{\cQ_{t_{0}}}\neq\ss{\cQ}$. If there is a smallest such $t_{0}$, then for each $\delta >0$ there exists $t\in [0,1]$ such that $|t-t_{0}|<\delta$ and $\ss{\cQ_{t}}\neq\ss{\cQ_{t_{0}}}$. Now $\sD^{b}(\bK Q)$ is Krull-Schmidt, so by \cite[Remark 3.8]{Co-stability} there exists $0<\epsilon_{t_{0}}<\frac{1}{2}$ such that, for all $\phi_{0}\in\bR$, within each interval $[\phi_{0}-\epsilon_{t_{0}},\phi_{0}+\epsilon_{t_{0}}]$ there is at most one $\cQ_{t_{0}}(\phi)\neq 0$. Similarly, there exists $0<\epsilon_{t}<\frac{1}{2}$ such that, for all $\phi_{0}\in\bR$, within each interval $[\phi_{0}-\epsilon_{t},\phi_{0}+\epsilon_{t}]$ there is at most one $\cQ_{t}(\phi)\neq 0$.

We claim that $d(\cQ_{t},\cQ_{t_{0}})\geq\epsilon_{0}=\min\{\epsilon_{t_{0}},\epsilon_{t}\}$. Suppose we have
\[
\inf\{\epsilon>0\,|\, \cQ_{t}(\phi)\subseteq\cQ_{t_{0}}([\phi-\epsilon,\phi+\epsilon]) \mbox{ for each } \phi\in\bR\} = d(\cQ_{t},\cQ_{t_{0}}) < \epsilon_{t_{0}}.
\]
This implies that if $0\neq\cQ_{t}(\phi)$ for some $\phi\in\bR$, then there exists $\psi\in[\phi-\epsilon_{t_{0}},\phi+\epsilon_{t_{0}}]$ so that 
\[
\cQ_{t}(\phi) \subseteq \cQ_{t_{0}}([\phi-\epsilon_{t_{0}},\phi+\epsilon_{t_{0}}])=\cQ_{t_{0}}(\psi).
\]
In particular, $\ss{\cQ_{t}}\subseteq \ss{\cQ_{t_{0}}}$. Analogously, if $d(\cQ_{t},\cQ_{t_{0}})<\epsilon_{t}$, then $\ss{\cQ_{t_{0}}}\subseteq \ss{\cQ_{t}}$. Thus, if $d(\cQ_{t},\cQ_{t_{0}})<\epsilon_{0}$, we have $\ss{\cQ_{t}}=\ss{\cQ_{t_{0}}}$; a contradiction.

This now contradicts the continuity of $\pi'$. Thus, there is no smallest such $t_{0}$. However, in this case, for each $\delta>0$, there exists $t\in (0,1]$ such that $t<\delta$ but  $\ss{\cQ_{t}}\neq\ss{\cQ}$. Arguing as above, we obtain the same contradiction to the continuity of $\pi'$. Hence, we must have $\ss{\cQ}=\ss{\cR}$.
\end{proof}

\begin{theorem} \label{thm:manifold}
The co-stability manifold of $\sD^{b}(\bK Q)$ is homeomorphic to $\bZ\cdot\bC^{2}$.
\end{theorem}

\begin{proof}
By Lemma \ref{lem:component}, the subsets $\sC(n)$ of Definitions \ref{def:component} are path components of $\Costab{\sD^{b}(\bK Q)}$. There are $\bZ$ such components $\sC(n)$, each determined by an indecomposable non-regular object $N_{n}$. As such, we only show that each component $\sC(n)$ is homeomorphic to $\bC^{2}$.

The map $f_{n}:\{(x,y)\in\bR^{2}\,|\,x<y\}\times\bR_{+}\times\bR_{+}\rightarrow\sC(n)$ in Lemma \ref{lem:continuous} is a continuous bijection. The inverse map $g_{n}:\sC(n)\rightarrow \{(x,y)\in\bR^{2}\,|\,x<y\}\times\bR_{+}\times\bR_{+}$, given by 
$$(Z,\cQ)\mapsto (\phi_{1},\phi_{0},|Z(N_{n+1}|,|Z(N_{n})|),$$
can also easily be seen to be continuous. Hence $f_{n}:\{(x,y)\in\bR^{2}\,|\,x<y\}\times\bR_{+}\times\bR_{+}\rightarrow \sC(n)$ is a homeomorphism. Now, on the left hand side, the topological space $\{(x,y)\in\bR^{2}\,|\,x<y\}\times\bR_{+}\times\bR_{+}$ is homeomorphic to $\bC^{2}$.
\end{proof}

\end{document}